\documentclass[12pt]{amsart}

\usepackage[hmargin=0.8in,height=8.6in]{geometry}
\usepackage{amssymb,amsthm, times}
\usepackage{delarray,verbatim}
\usepackage{ifpdf}
\ifpdf
\usepackage[pdftex]{graphicx}
 \else
\usepackage[dvips]{graphicx}
 \fi

\usepackage{bm}

\usepackage{ifpdf}
\usepackage{color}
\definecolor{webgreen}{rgb}{0,.5,0}
\definecolor{webbrown}{rgb}{.8,0,0}
\definecolor{emphcolor}{rgb}{0.95,0.95,0.95}

\usepackage{hyperref}
\hypersetup{%
          colorlinks=true,
          linkcolor=webbrown,
          filecolor=webbrown,
          citecolor=webgreen,
          breaklinks=true}
\ifpdf \hypersetup{pdftex,
            pdfstartview=FitH, 
            bookmarksopen=true,
            bookmarksnumbered=true
} \else \hypersetup{dvips} \fi

\newcommand {\ud}{{\rm d}}

\numberwithin{equation}{section}

\newtheorem{theorem}{Theorem}[section]
\newtheorem{proposition}{Proposition}[section]
\newtheorem{corollary}{Corollary}[section]
\newtheorem{remark}{Remark}[section]
\newtheorem{lemma}{Lemma}[section]

\numberwithin{remark}{section} \numberwithin{proposition}{section}
\numberwithin{corollary}{section}
\newcommand {\R}{\mathbb{R}}

\newcommand {\p}{\mathbb{P}}

\newcommand {\E}{\mathbb{E}}

\newcommand{\diff}{{\rm d}}

\newcommand{\lev}{L\'{e}vy }

\newcommand{\e}{\mathbb{E}}

\begin{document}
	\title{Mixed Periodic-Classical Barrier Strategies for L\'evy Risk Processes}
	
	\thanks{This version: \today.  J. L. P\'erez  is  supported  by  CONACYT,  project  no.\ 241195.
K. Yamazaki is in part supported by MEXT KAKENHI grant no.\  26800092.}
\author[J. L. P\'erez]{Jos\'e-Luis P\'erez$^*$}
\thanks{$*$\, Department of Probability and Statistics, Centro de Investigaci\'on en Matem\'aticas A.C. Calle Jalisco s/n. C.P. 36240, Guanajuato, Mexico. Email: jluis.garmendia@cimat.mx.  }
\author[K. Yamazaki]{Kazutoshi Yamazaki$^\dag$}
\thanks{$\dag$\, Department of Mathematics,
Faculty of Engineering Science, Kansai University, 3-3-35 Yamate-cho, Suita-shi, Osaka 564-8680, Japan. Email: kyamazak@kansai-u.ac.jp.   }
\date{}

	\maketitle

		\begin{abstract} Given a spectrally negative \lev process and independent Poisson observation times, we consider a periodic barrier strategy that pushes the process down to a certain level whenever it is above it.  We also consider the versions with additional classical reflection above and/or below.
Using scale functions and excursion theory, various fluctuation identities are computed in terms of the scale function.  Applications in de Finetti's dividend problems are also discussed.
\\
\noindent \small{\noindent  AMS 2010 Subject Classifications: 60G51, 91B30 }
	\end{abstract}
	
\section{Introduction}

In actuarial risk theory, the surplus of an insurance company is typically modeled by a compound Poisson process with a positive drift and negative jumps (Cram\'er-Lundberg model) or more generally by a spectrally negative \lev process. Thanks to the recent developments of the fluctuation theory of \lev processes, there now exist a variety of tools available to compute various quantities that are useful in insurance mathematics. 

By the existing fluctuation theory, it is relatively easy to deal with (classical) reflected \lev processes that can be written as the differences between the underlying and running supremum/infimum processes. 

The known results on these processes can be conveniently and efficiently applied in modeling the surplus of a dividend-paying company: under \emph{a barrier strategy}, the resulting controlled surplus process becomes the process reflected from above.  Avram et al.\ \cite{APP2007} obtained the expected net present value (NPV) of dividends until ruin; a sufficient condition for the optimality of a barrier strategy is given in Loeffen \cite{Loeffen2008}.  Similarly, capital injection is modeled by reflections from below.  In the bail-out case with a requirement that ruin must be avoided, Avram et al.\ \cite{APP2007} obtained the expected NPV of dividends and capital injections under a double barrier strategy. They also showed that it is optimal to reflect the process at $0$ and at some upper boundary, with the resulting surplus process being a doubly reflected \lev process.

These seminal works give concise expressions for various fluctuation identities in terms of the \emph{scale function}.  In general, conciseness is still maintained when the underlying spectrally one-sided \lev process is replaced with its reflected process.  This is typically done by using the derivative or the integral of the scale function depending on whether the reflection barrier is upper or lower.

In this paper, we consider a different version of reflection, which we call the \emph{Parisian reflection}.  Motivated by the fact that, in reality, dividend/capital injection decisions can only be made at some intervals, several recent papers consider \emph{periodic barrier strategies} that reflect the process only at discrete observation times.  In particular, Avram et al.\ \cite{APY} consider, for a general spectrally negative \lev process, the case capital injections can be made at the jump times of an independent Poisson process (reflection barrier is lower). This current paper considers the case when dividends are made at these Poisson observation times (reflection barrier is upper).  Other related papers in the compound Poisson cases include \cite{albrecher2011randomized}  and
\cite{avanzi2013periodic}, where in the former several identities are obtained when the solvency is also observed periodically whereas the latter studies the case where observation intervals are Erlang-distributed.   

This work is also motivated by its potential applications in de Finetti's dividend problems under Poisson observation times.   In the dual (spectrally positive) model, Avanzi et al.\ \cite{ATW2014} solved the case where the  jump size is hyper-exponentially distributed; P\'erez and Yamazaki \cite{PY16a} recently generalized the results to a general spectrally positive \lev case and also solved the bail-out version using the results in \cite{APY}.  An extension with a combination of 
periodic and continuous dividend payments (with different transaction costs) is recently  solved by Avanzi et al.\ \cite{ATW2016} when the underlying process is a Brownian motion with a drift.  In these papers, optimal strategies are of \emph{periodic barrier-type}.  To our best knowledge, these problems are not solved for a general spectrally negative \lev case:  our aim in this paper is to  give concise expressions for the expected NPVs under periodic barrier strategies, which can be reasonably conjectured to be optimal solutions.

In this paper, we study the following four processes that are constructed from a given spectrally negative \lev process $X$ and the jump times of an independent Poisson process with rate $r > 0$:
\begin{enumerate}
\item \underline{The process with Parisian reflection from above $X_r$}:  The process  $X_r$ is constructed by modifying $X$ so that it is pushed down to zero at the Poisson observation times at which it is above zero.   Note that the barrier level $0$ can be changed to any real value by the spatial homogeneity of $X$.  This process models the controlled surplus process under a periodic barrier dividend strategy. 
\item \underline{The process with Parisian and classical reflection from above $\tilde{X}_r^b$}: 
Suppose $\overline{Y}^b$ is the reflected process of $X$ with the classical upper barrier $b > 0$. The process $\tilde{X}_r^b$ is constructed in the same way as $X_r$ in (1) with the underlying process $X$ replaced with $\overline{Y}^b$. This process models the controlled surplus process under a combination of a classical and periodic barrier dividend strategies.  This is a generalization of the Brownian motion case as studied in \cite{ATW2016}.
\item \underline{The process with Parisian reflection from above and  classical reflection from below $Y_r^a$}: Suppose $\underline{Y}^a$ is the reflected process of $X$ with the classical lower barrier $a < 0$. The process $Y_r^a$ is constructed in the same way as $X_r$ as in (1) with the underlying process $X$ replaced with $\underline{Y}^a$. By shifting the process (by $-a$), it models the surplus under a periodic barrier dividend strategy with classical capital injections (so that it does not go below zero).
\item \underline{The process with Parisian and classical reflection from above and  classical reflection from below $\tilde{Y}_r^{a,b}$}: Suppose $Y^{a,b}$ is the doubly reflected process of $X$  with a classical lower barrier $a < 0$ and a classical upper barrier $b > 0$. The process $\tilde{Y}_r^{a,b}$ is constructed in the same way as $X_r$ in (1) with the underlying process $X$ replaced with $Y^{a,b}$. By shifting the process (by $-a$), it models the controlled surplus process under a combination of a classical and periodic barrier dividend strategies as in (2) with additional classical capital injections.
\end{enumerate}

For these four processes, we compute various fluctuation identities that include
\begin{enumerate}
\item[(a)] the expected NPV of dividends (both corresponding to Parisian and classical reflections) with the horizon given by the first exit time from an interval and those with the infinite horizon, 
\item[(b)] the expected NPV of capital injections with the horizon given by the first exit time from an interval and those with the infinite horizon, 
\item[(c)] the two-sided (one-sided) exit identities.
 \end{enumerate}
	
In order to compute these for the four processes defined above, we first obtain the identities for the process (1) killed upon exiting $[a,b]$. Using the observation that the paths of the processes (2)-(4) are identical to those of (1) before the first exit time from $[a,b]$, the results for (2)-(4) can be obtained as corollaries, via the strong Markov property and the existing known identities for classical reflected processes.

The identities for (1) are obtained separately for the case $X$ is of bounded variation and for the case it is of unbounded variation.  The former is done by a relatively well-known technique via the strong Markov property combined with the existing known identities for the spectrally negative \lev process. The case of unbounded variation is done via excursion theory (in particular excursions away from zero as in \cite{PPR15b}).  Thanks to the simplifying formulae obtained in \cite{APY} and \cite{LRZ}, concise expressions can be achieved.

The rest of the paper is organized as follows.  In Section \ref{section_model}, we review the spectrally negative \lev process and construct more formally the four processes described above.  In addition, scale functions and some existing fluctuation identities are briefly reviewed.  In Section \ref{section_results1}, we state the main results for the process (1), and then in Section \ref{section_results2} those for the processes (2)-(4).  In Sections \ref{section_proof_bounded} and \ref{section_proof_unbounded}, we show the main results for (1) for the case of bounded variation and unbounded variation, respectively.

 Throughout the paper, for any function $f$ of two variables, let $f'(\cdot, \cdot)$ be the partial derivative with respect to the first argument.

	\section{Spectrally negative L\'evy processes with Parisian reflection above} \label{section_model}
	Let $X=(X(t); t\geq 0)$ be a L\'evy process defined on a  probability space $(\Omega, \mathcal{F}, \p)$.  For $x\in \R$, we denote by $\p_x$ the law of $X$ when it starts at $x$ and write for convenience  $\p$ in place of $\p_0$. Accordingly, we shall write $\e_x$ and $\e$ for the associated expectation operators. In this paper, we shall assume throughout that $X$ is \textit{spectrally negative},   meaning here that it has no positive jumps and that it is not the negative of a subordinator.  It is a well known fact that its Laplace exponent $\psi:[0,\infty) \to \R$, i.e.
	\[
	\e\big({\rm e}^{\theta X(t)}\big)=:{\rm e}^{\psi(\theta)t}, \qquad t, \theta\ge 0,
	\] 
	is given, by the \emph{L\'evy-Khintchine formula}
	\begin{equation}\label{lk}
	\psi(\theta):=\gamma\theta+\frac{\sigma^2}{2}\theta^2+\int_{(-\infty,0)}\big({\rm e}^{\theta x}-1-\theta x\mathbf{1}_{\{x > -1\}}\big)\Pi(\ud x), \quad \theta \geq 0,
	\end{equation}
	where $\gamma \in \R$, $\sigma\ge 0$, and $\Pi$ is a measure on $(-\infty, 0)$ called the L\'evy measure of $X$ that satisfies
	\[
	\int_{(-\infty,0)}(1\land x^2)\Pi(\ud x)<\infty.
	\]
	

	It is well-known that $X$ has paths of bounded variation if and only if $\sigma=0$ and $\int_{(-1,0)} |x|\Pi(\mathrm{d}x) < \infty$; in this case, $X$ can be written as
	\begin{equation}
	X(t)=ct-S(t), \,\,\qquad t\geq 0,\notag
	\end{equation}
	where 
	\begin{align*}
		c:=\gamma-\int_{(-1,0)} x\Pi(\mathrm{d}x) 
	\end{align*}
	and $(S(t); t\geq0)$ is a driftless subordinator. Note that  necessarily $c>0$, since we have ruled out the case that $X$ has monotone paths; its Laplace exponent is given by
	\begin{equation*}
		\psi(\theta) = c \theta+\int_{(-\infty,0)}\big( {\rm e}^{\theta x}-1\big)\Pi(\ud x), \quad \theta \geq 0.
	\end{equation*}
	
Let us define  the \emph{running infimum and supremum processes} 
\begin{align*}
\underline{X}(t) := \inf_{0 \leq t' \leq t} X(t') \quad \textrm{and} \quad \overline{X}(t) := \sup_{0 \leq t' \leq t} X(t'), \quad  t \geq 0.
\end{align*}
Then, the processes reflected from above at $b$ and below at $a$ are given, respectively, by
\begin{align}
\overline{Y}^b(t) &:= X(t) - L^b(t) \quad \textrm{and} \quad \underline{Y}^a(t) := X(t) + R^a(t), \quad  t \geq 0, \label{classical_reflected}
\end{align}
where
\begin{align}
L^b(t) := (\overline{X}(t) -b) \vee 0 \quad \textrm{and} \quad R^a(t) := (a - \underline{X}(t)) \vee 0, \quad t \geq 0, \label{classical_decomp_reflected}
\end{align}
are the cumulative amounts of reflections that push the processes downward and upward, respectively.

	\subsection{\lev processes with Parisian reflection above} \label{subsection_process_defined}
	
	
	
	Let $\mathcal{T}_r=\{T(i); i \geq 1\}$ be an increasing sequence of jump times of an independent Poisson process with rate $r >0$.   We construct the \emph{\lev process with Parisian reflection above} $X_r = (X_r(t); t \geq 0)$ as follows: the process is only observed at times $\mathcal{T}_r$ and is pushed down to $0$ if only if it is above $0$.
	
	More specifically, we have
	\begin{align*} 
	X_r(t) = X(t), \quad 0 \leq t < T_0^+(1),
	\end{align*}
	where
	\begin{align} T_{0}^+(1) := \inf\{T(i):\; X(T(i)) > 0\}; \label{def_T_0_1}
	\end{align}
	here and throughout, let $\inf \varnothing = \infty$.
	The process then jumps downward by $X(T_0^+(1))$ so that $X_r(T_0^+(1)) = 0$. For $T_0^+(1) \leq t < T_0^+(2)  := \inf\{T(i) > T_0^+(1):\; X_r(T(i)-) > 0\}$, we have $X_r(t) = X(t) -X(T_0^+(1))$, and $X_r(T_0^+(2)) = 0$.  The process can be constructed by repeating this procedure.
	
	Suppose $L_r(t)$ is the cumulative amount of (Parisian) reflection until time $t \geq 0$. Then we have
	\begin{align}
	X_r(t) = X(t) - L_r(t), \quad t \geq 0, \label{X_r_decomposition}
	\end{align}
	with
	\begin{align*}
	L_r(t) := \sum_{T^+_0(i) \leq t} X_r(T_0^+(i)-), \quad t \geq 0, 
	\end{align*}
	where $(T_{0}^+(n); n \geq 1)$ can be constructed inductively by \eqref{def_T_0_1} and
	\begin{eqnarray*} T_{0}^+(n+1) := \inf\{T(i) > T_0^+(n):\; X_r(T(i)-) > 0\}, \quad n \geq 1.
	\end{eqnarray*}
	
\subsection{\lev processes with Parisian and classical reflection above}  \label{subsection_parisian_classical_above}
Fix $b > 0$.  Consider an extension of the above with additional classical reflection from above at $b > 0$, which we denote by $\tilde{X}_r^b$. More specifically, we have
	\begin{align*} 
	\tilde{X}_r^b(t) = \overline{Y}^{b}(t), \quad 0 \leq t < \tilde{T}_0^+(1),
	\end{align*}
	where $\tilde{T}_{0}^+(1) := \inf\{T(i):\; \overline{Y}^{b}(T(i)) > 0\}$.
	The process then jumps downward by $\overline{Y}^{b}(\tilde{T}_0^+(1))$ so that $\tilde{X}_r^b(\tilde{T}_0^+(1)) = 0$. For $\tilde{T}_0^+(1) \leq t < \tilde{T}_0^+(2)  := \inf\{T(i) > \tilde{T}_0^+(1):\; \tilde{X}_r^b(T(i)-) > 0\}$, it is the reflected process of $X(t) -X(\tilde{T}_0^+(1))$ (with classical reflection above at $b$ as in \eqref{classical_reflected}), and $\tilde{X}_r^b(\tilde{T}_0^+(2)) = 0$.   The process can be constructed by repeating this procedure.

	Suppose $\tilde{L}_{r, P}^b(t)$ and $\tilde{L}_{r,S}^b(t)$ are the cumulative amounts of Parisian reflection (with upper barrier 0) and classical reflection (with upper barrier $b$) until time $t \geq 0$. Then we have
	\begin{align*}
	\tilde{X}_r^b(t) = X(t) - \tilde{L}_{r, P}^b(t) - \tilde{L}_{r,S}^b(t), \quad t \geq 0.
	\end{align*}
	
\subsection{\lev processes with Parisian reflection above and classical reflection below}  \label{subsection_with_reflection_below}
	Fix $a < 0$. The process $Y_r^a$ with additional (classical) reflection below can be defined analogously.   
	We have
	\begin{align*}
	Y^a_r(t) = \underline{Y}^a(t), \quad 0 \leq t < \widehat{T}_0^{+} (1) 
	\end{align*}
	where $\widehat{T}_{0}^+(1) := \inf\{T(i):\; \underline{Y}^a(T(i)) > 0\}$.
	The process then jumps downward by $\underline{Y}^a(\widehat{T}_0^+(1))$ so that $Y^a_r(\widehat{T}_0^+(1)) = 0$. For $\widehat{T}_0^+(1) \leq t < \widehat{T}_0^+(2)  := \inf\{T(i) > \widehat{T}_0^+(1):\; Y^a_r(T(i) -) > 0\}$, $Y_r^a(t)$ is the reflected process of $X(t) - X(\widehat{T}_0^+(1))$ (with the classical reflection below at $a$ as in \eqref{classical_reflected}), and $Y_r^a(\widehat{T}_0^+(2)) = 0$.  
	The process can be constructed by repeating this procedure.
	It is clear that it admits a decomposition
	\begin{align*}
	Y^a_r(t) = X(t)   -L_r^a(t) + R_r^a (t), \quad t \geq 0,
	\end{align*}
	where $L_r^a(t)$ and $R_r^a(t)$ are, respectively, the cumulative amounts of Parisian reflection (with upper barrier $0$) and classical reflection (with lower barrier $a$) until time $t$.
	
\subsection{\lev processes with Parisian and classical reflection above and classical reflection below}  \label{subsection_double_reflected_case}
Fix $a < 0 < b$.
Consider a version of $Y_r$ with additional classical reflection from above at $b > 0$. More specifically, we have 
	\begin{align*} 
	\tilde{Y}_r^{a,b}(t) = Y^{a,b}(t), \quad 0 \leq t < \check{T}_0^+(1),
	\end{align*}
	where $Y^{a,b}$ is the classical doubly reflected process of $X$ with lower barrier $a$ and upper barrier $b$ (see Pistorius \cite{P2003}) and
	\begin{align*} \check{T}_{0}^+(1) := \inf\{T(i):\; Y^{a,b}(T(i)) > 0\}.
	\end{align*}
	The process then jumps downward by $Y^{a,b}(\check{T}_0^+(1))$ so that $\tilde{Y}_r^{a,b}(\check{T}_0^+(1)) = 0$. For $\check{T}_0^+(1) \leq t < \check{T}_0^+(2)  := \inf\{T(i) > \check{T}_0^+(1):\; \tilde{Y}_r^{a,b}(T(i)-) > 0\}$, it is the doubly reflected process of $X(t) -X(\check{T}_0^+(1))$ (with classical reflections at $a$ and $b$),  and $\tilde{Y}_r^{a,b}(\check{T}_0^+(2)) = 0$.  The process can be constructed by repeating this procedure.

	Suppose $\tilde{L}_{r, P}^{a,b}(t)$ and $\tilde{L}_{r,S}^{a,b}(t)$ are the cumulative amounts of Parisian reflection (with upper barrier $0$) and classical reflection (with upper barrier $b$) until time $t \geq 0$, and $\tilde{R}_r^{a,b} (t)$ is that of the classical reflection (with lower barrier $a$). Then we have
	\begin{align*}
	\tilde{Y}_r^{a,b}(t) = X(t)- \tilde{L}_{r, P}^{a,b}(t) - \tilde{L}_{r,S}^{a,b}(t) + \tilde{R}_r^{a,b}(t), \quad t \geq 0.
	\end{align*}	
	\subsection{Review on scale functions.}
	 
	 Fix $q \geq0$. We use $W^{(q)}$ for the scale function of the spectrally negative \lev process $X$.  This is the mapping from $\R$ to $[0, \infty)$ that takes value zero on the negative half-line, while on the positive half-line it is a strictly increasing function that is defined by its Laplace transform:
	 \begin{align} \label{scale_function_laplace}
	 	\begin{split}
	 		\int_0^\infty  \mathrm{e}^{-\theta x} W^{(q)}(x) \diff x &= \frac 1 {\psi(\theta)-q}, \quad \theta > \Phi(q),
	 	\end{split}
	 \end{align}
	 where $\psi$ is as defined in \eqref{lk} and
	 \begin{align*}
	 	\begin{split}
	 		\Phi(q) := \sup \{ \lambda \geq 0: \psi(\lambda) = q\} . 
	 	\end{split}
	 \end{align*}
	 We also define, for $x \in \R$, 
	 \begin{align*}
	 	\overline{W}^{(q)}(x) &:=  \int_0^x W^{(q)}(y) \diff y, \qquad
		\overline{\overline{W}}^{(q)}(x) :=\int_0^x \int_0^z W^{(q)} (w) \diff w \diff z, \\
	 	Z^{(q)}(x) &:= 1 + q \overline{W}^{(q)}(x), \qquad
	 	\overline{Z}^{(q)}(x) := \int_0^x Z^{(q)} (z) \diff z = x + q \overline{\overline{W}}^{(q)}(x).
	 \end{align*}
	 Noting that $W^{(q)}(x) = 0$ for $-\infty < x < 0$, we have
	 \begin{align*}
	 	\overline{W}^{(q)}(x) = 0, \quad \overline{\overline{W}}^{(q)}(x) = 0, \quad Z^{(q)}(x) = 1,
	 	\quad \textrm{and} \quad \overline{Z}^{(q)}(x) = x, \quad x \leq 0. 
	 \end{align*}
Define also
\begin{align*} 
Z^{(q)}(x, \theta ) &:=e^{\theta x} \left( 1 + (q- \psi(\theta )) \int_0^{x} e^{-\theta  z} W^{(q)}(z) \diff z	\right), \quad x \in \R, \, \theta  \geq 0,
\end{align*}
and its partial derivative with respect to the first argument:
\begin{align} \label{Z_theta_derivative}
Z^{(q) \prime}(x, \theta ) &= \theta Z^{(q) }(x, \theta )  + (q-\psi(\theta)) W^{(q)}(x), \quad x \in \R, \, \theta  \geq 0.
\end{align}
In particular, for $x \in \R$, $Z^{(q)}(x, 0) =Z^{(q)}(x)$ and, for $r > 0$,
\begin{align*} 
\begin{split}
Z^{(q)}(x, \Phi(q+r)) &:=e^{\Phi(q+r) x} \left( 1 -r \int_0^{x} e^{-\Phi(q+r) z} W^{(q)}(z) \diff z \right),  \\
Z^{(q+r)}(x, \Phi(q)) &:=e^{\Phi(q) x} \left( 1 + r \int_0^{x} e^{-\Phi(q) z} W^{(q+r)}(z) \diff z	\right).
\end{split}
\end{align*}

	 
	 

	 \begin{remark} \label{remark_smoothness_zero}
	 	\begin{enumerate}
	 		\item If $X$ is of unbounded variation or the \lev measure is atomless, it is known that $W^{(q)}$ is $C^1(\R \backslash \{0\})$; see, e.g.,\ \cite[Theorem 3]{Chan2011}. In particular, if $\sigma > 0$, then $W^{(q)}$ is $C^2(\R \backslash \{0\})$; see, e.g.,\ \cite[Theorem 1]{Chan2011}.
	 		\item Regarding the asymptotic behavior near zero, as in Lemmas 3.1 and 3.2 of \cite{KKR},
	 		\begin{align}\label{eq:Wqp0}
	 			\begin{split}
	 				W^{(q)} (0) &= \left\{ \begin{array}{ll} 0 & \textrm{if $X$ is of unbounded
	 						variation,} \\ \frac 1 {c} & \textrm{if $X$ is of bounded variation,}
	 				\end{array} \right. \\
	 				W^{(q)\prime} (0+) &:= \lim_{x \downarrow 0}W^{(q)\prime} (x) =
	 				\left\{ \begin{array}{ll}  \frac 2 {\sigma^2} & \textrm{if }\sigma > 0, \\
	 					\infty & \textrm{if }\sigma = 0 \; \textrm{and} \; \Pi(-\infty,0)= \infty, \\
	 					\frac {q + \Pi(-\infty,0)} {c^2} &  \textrm{if }\sigma = 0 \; \textrm{and} \; \Pi(-\infty,0) < \infty.
	 				\end{array} \right.
	 			\end{split}
	 		\end{align}
	 		On the other hand, as in Lemma 3.3 of \cite{KKR},
	 		\begin{align}
	 			\begin{split}
	 				e^{-\Phi(q) x}W^{(q)} (x) \nearrow \psi'(\Phi(q))^{-1}, \quad \textrm{as } x \uparrow \infty,
	 			\end{split}
	 			\label{W^{(q)}_limit}
	 		\end{align}
	 		where in the case $\psi'(0+) = 0$, the right hand side, when $q=0$,  is understood to be infinity.	 		
			
	 	\end{enumerate}
		
		Below, we list the fluctuation identities that will be used later in the paper.  
	 \end{remark}
	 \subsection{Fluctuation identities for $X$}
	 	 Let
	 \begin{align*}
	 \tau_a^- := \inf \left\{ t \geq 0: X(t) < a \right\} \quad \textrm{and} \quad \tau_b^+ := \inf \left\{ t \geq 0: X(t) >  b \right\}, \quad a, b \in \R.
	 \end{align*}
Then for $b > a$ and $x \leq b$, 
	 \begin{align}
	 	\begin{split}
	 		\E_x \left( e^{-q \tau_b^+}; \tau_b^+ < \tau_a^- \right) &= \frac {W^{(q)}(x-a)}  {W^{(q)}(b-a)}, \\ \E_x \left( e^{-q \tau_a^- - \theta  [a-X(\tau_a^-)]}; \tau_b^+ > \tau_a^- \right) &= Z^{(q)}(x-a,\theta ) -  Z^{(q)}(b-a,\theta ) \frac {W^{(q)}(x-a)}  {W^{(q)}(b-a)}, \quad \theta  \geq 0.
	 	\end{split}
	 	\label{laplace_in_terms_of_z}
	 \end{align}
By taking $b \uparrow \infty$ in the latter, as in \cite[(7)]{AIZ} (see also the identity (3.19) in \cite{APP2007}), 
\begin{align*}
&&\E_x \big( e^{-q \tau_a^- -\theta  [a-X(\tau_a^-)]} ; \tau_a^- < \infty\big) = Z^{(q)}(x-a,\theta ) - W^{(q)}(x-a)  \frac {\psi(\theta )-q}{ \theta -\Phi(q)},
\end{align*}
where, for the case $\theta  = \Phi(q)$, it is understood as the limiting case.
In addition, it is known that a spectrally negative \lev process creeps downwards if and only if $\sigma>0$; by Theorem 2.6 (ii) of \cite{KKR}, 
		\begin{align} \label{creeping_identity}
\E_x \big( e^{-q \tau_a^-}; X(\tau_a^-) = a, \tau_a^- < \infty\big)
		=\frac{\sigma^2}{2}\left[ W^{(q)\prime}(x-a)-\Phi(q) W^{(q)}(x-a)\right], \quad x > a,
		\end{align}
		where we recall that $W^{(q)}$ is differentiable when $\sigma > 0$ as in Remark \ref{remark_smoothness_zero} (1).
		By this, the strong Markov property, and \eqref{laplace_in_terms_of_z}, we have for $a < b$ and $x \leq b$,
		\begin{align}
		\begin{split}
&\E_x (e^{-q \tau_{a}^-}; X(\tau_a^-) = a,\tau_a^- < \tau_b^+) \\&= \E_x (e^{-q \tau_{a}^-}; X(\tau_a^-) = a, \tau_a^- < \infty)  - \E_x (e^{-q \tau_{b}^+}; \tau_b^+ < \tau_a^- ) \E_b (e^{-q \tau_{a}^-}; X(\tau_a^-) = a, \tau_a^- < \infty)  \\  &=  C_{b-a}^{(q)}(x-a)
\label{creeping_two_sided}
\end{split}
\end{align}
where  
\begin{align*} 
C_{\beta}^{(q)}(y) :=\frac {\sigma^2} 2\left( W^{(q)\prime}(y) - \frac {W^{(q)}(y)} {W^{(q)}(\beta)} W^{(q)\prime} (\beta)\right), \quad y \in \R \backslash \{0\}, \; \beta > 0.
\end{align*}
\subsection{Fluctuation identities for $\overline{Y}^b(t)$}
Fix $a < b$.
Define  the first down-crossing time of $\overline{Y}^b(t)$ of \eqref{classical_reflected}: 
\begin{align}
\tilde{\tau}_{a,b}^- := \inf \{ t > 0: \overline{Y}^b(t) < a\}. \label{downcrossing_classical_reflected}
\end{align}
The Laplace transform of $\tilde{\tau}_{a,b}^-$ is given, as in  Proposition 2 (ii) of \cite{P2004}, by
\begin{align}
\E_x ( e^{- q \tilde{\tau}_{a,b}^-} ) =Z^{(q)}(x-a) - q W^{(q)}(b-a)  \frac {W^{(q)}(x-a)} {W^{(q)\prime}((b-a)+)}, \quad q \geq 0, \; x \leq b. \label{downcrossing_time_reflected}
\end{align}
As in Proposition 1 of \cite{APP2007}, the discounted cumulative amount of reflection from above as in \eqref{classical_decomp_reflected} is
\begin{align}
\E_x \Big( \int_{[0,\tilde{\tau}_{a,b}^-]} e^{-qt} \diff L^b(t) \Big) = \frac {W^{(q)}(x-a)} {W^{(q)\prime}((b-a)+)}, \quad q \geq 0, \; x \leq b. \label{dividend_classical_barrier}
\end{align}
\subsection{Fluctuation identities for $\underline{Y}^a(t)$}
Fix $a < b$. Define  the first up-crossing time of $\underline{Y}^a(t)$  of \eqref{classical_reflected}:
\begin{align}
\eta^+_{a,b} := \inf \{ t > 0: \underline{Y}^a(t) > b\}. \label{eta_a_b_plus}
\end{align}
First, as in page 228 of \cite{K}, its Laplace transform is concisely given by
\begin{align} \label{upcrossing_time_reflected}
\E_x ( e^{- q \eta^+_{a,b}} ) = \frac {Z^{(q)}(x-a)} {Z^{(q)}(b-a)}, \quad q \geq 0, \; x \leq b.
\end{align}
Second, as in the proof of Theorem 1 of \cite{APP2007}, the discounted cumulative amount of reflection from below as in \eqref{classical_decomp_reflected} is, given $\psi'(0+) > -\infty$, 
\begin{align} \label{overshoot_classical_expectation}
\E_x \Big( \int_{[0,\eta^+_{a,b}]} e^{-qt} \diff  R^a(t) \Big) =- l^{(q)}(x-a) + \frac {Z^{(q)}(x-a)} {Z^{(q)}(b-a)}  l^{(q)}(b-a), \quad  q \geq 0, \; x \leq b,
\end{align}
where
	 \begin{align*}
	 l^{(q)}(x) := \overline{Z}^{(q)} (x) - \psi'(0+) \overline{W}^{(q)}(x), \quad q \geq 0, \; x \in \R.
	 \end{align*}
	 
\subsection{Some more notations}
	 
	 For the rest of the paper, we fix $r > 0$, and use $\mathbf{e}_r$ for  the first observation time, or an independent exponential random variable with parameter $r$.

Let, for $q \geq 0$ and $x \in \R$,
\begin{align} \label{Z_q_r}
\begin{split}
\tilde{Z}^{(q,r)}(x,\theta ) &:= \frac {rZ^{(q)}(x,\theta )+(q-\psi(\theta ))Z^{(q)}(x,\Phi(q+r))} {\Phi(q+r) - \theta}, \quad \theta  \geq 0, \\
\tilde{Z}^{(q,r)}(x)&:=\tilde{Z}^{(q,r)}(x, 0) = \frac {rZ^{(q)}(x)+qZ^{(q)}(x,\Phi(q+r))} {\Phi(q+r)},
\end{split}
\end{align}
where the case $\theta = \Phi(q+r)$ is understood as the limiting case.

We define, for any measurable function $f: \R \to \R$,
\begin{align}
\mathcal{M}^{(q,r)}_a f(x) &:= f (x-a) +r \int_0^x W^{(q+r)} (x-y) f(y-a) \diff y, \quad x \in \R, \quad a < 0. \label{operator_M}
\end{align}
In particular, we let, for $a < 0$, $q \geq 0$, and $x \in \R$,  
\begin{align*}
W^{(q,r)}_a(x) &:=  \mathcal{M}^{(q,r)}_a W^{(q)}(x), \quad \overline{W}^{(q,r)}_a(x) :=  \mathcal{M}^{(q,r)}_a \overline{W}^{(q)}(x),  \\ Z^{(q,r)}_a(x, \theta ) &:=  \mathcal{M}^{(q,r)}_a Z^{(q)}(x, \theta ), \; \theta  \geq 0, \quad \overline{Z}^{(q,r)}_a(x) :=  \mathcal{M}^{(q,r)}_a \overline{Z}^{(q)}(x),
\end{align*}
with $Z^{(q,r)}_a(\cdot) := Z^{(q,r)}_a(\cdot, 0)$.


Thanks to these functionals, the following expectations admit concise expressions. 
By Lemma 2.1 in \cite{LoRZ} and Theorem 6.1 in \cite{APY}, for all $q \geq 0$, $a < 0 < b$, and $x \leq b$,
\begin{align} 
\E_x \big(e^{-(q+r) \tau_0^-} W^{(q)}(X(\tau_0^-)-a); \tau_0^- < \tau_b^+ \big) &= W^{(q,r)}_a(x)  - \frac {W^{(q+r)}(x)} {W^{(q+r)}(b)} W^{(q,r)}_a(b), \label{scale_function_overshoot_simplifying} \\
\E_x \Big( e^{-(q+r) \tilde{\tau}_{0,b}^-} W^{(q)}(\overline{Y}^b(\tilde{\tau}_{0,b}^-)-a)  \Big)&=W^{(q,r)}_a (x)-\frac{W^{(q+r)}(x)}{W^{(q+r) \prime}(b+)} (W^{(q,r)}_a)' (b +). \label{scale_function_overshoot_simplifying_reflected}
\end{align}
In addition, we give a slight generalization of Lemma 2.1 of \cite{LRZ} and Theorem 6.1 in \cite{APY}.  The proofs are  given in Appendix \ref{proof_simplifying_formula}.
\begin{lemma} \label{lemma_simplifying_formula_measure_changed}
For $q \geq 0$, $\theta  \geq 0$, $a < 0 < b$, and $x \leq b$,
\begin{align}
\E_x \big(e^{-(q+r) \tau_0^-} Z^{(q)} (X(\tau_0^-)-a,\theta ); \tau_0^- < \tau_b^+ \big) 
&= Z^{(q,r)}_a(x,\theta ) - \frac {W^{(q+r)}(x)} {W^{(q+r)}(b)}  Z^{(q,r)}_a(b,\theta ), \label{Z_overshoot_identity} \\
\E_x \big(e^{-(q+r) \tilde{\tau}_{0,b}^-} Z^{(q)} (\overline{Y}^b(\tilde{\tau}_{0,b}^-)-a, \theta )  \big) &=  Z_a^{(q,r)}(x,\theta ) - \frac {W^{(q+r)}(x)} {W^{(q+r)\prime} (b+)} (Z_a^{(q,r)})'(b,\theta ).  \label{Z_overshoot_identity_reflected}\end{align}
\end{lemma}


	 \section{Main results for $X_r$} \label{section_results1}
In this section, we obtain the fluctuation identities for the process $X_r$ as constructed in Section \ref{subsection_process_defined}.  The main theorems are obtained for the case killed upon exiting an interval $[a,b]$ for $a < 0 < b$.  As their corollaries, we also obtain the limiting cases as $a \downarrow -\infty$ and $b \uparrow \infty$. The proofs for the theorems are given in Sections \ref{section_proof_bounded} and \ref{section_proof_unbounded} for the bounded and unbounded variation cases, respectively.  The proofs for the corollaries are given in the appendix.

Define the first down/up-crossing times for $X_r$,
\begin{align*}
\tau_a^-(r) := \inf \{ t > 0: X_r(t)  < a\} \quad \textrm{and} \quad \tau_b^+(r) := \inf \{ t > 0: X_r(t) > b\}, \quad a, b \in \R.
\end{align*}
Define also for $q \geq 0$, $a < 0$, and $x \in \R$,
\begin{align} \label{def_I}
\begin{split}
I_a^{(q,r)} (x) &:= \frac {W^{(q,r)}_a(x)} {W^{(q)}(-a)}  - r  \overline{W}^{(q+r)} (x), \\
J_a^{(q,r)}(x,\theta ) &:=Z^{(q,r)}_a(x,\theta ) - r Z^{(q)} (-a, \theta ) \overline{W}^{(q+r)} (x),  \\
J_a^{(q,r)}(x) &:= J_a^{(q,r)}(x,0) = Z^{(q,r)}_a(x) - r Z^{(q)} (-a) \overline{W}^{(q+r)} (x).  
\end{split}
\end{align}
Note in particular
	\begin{align}
I_a^{(q,r)} (0) = 1 \quad \textrm{and} \quad J_a^{(q,r)}(0,\theta ) = Z^{(q)}(-a, \theta), \label{I_J_zero}
\end{align}
and that
\begin{align}
J_a^{(0,r)}(x) = 1 \quad \textrm{and} \quad (J_a^{(0,r)})'(x ) = 0, \quad x \in \R. \label{J_tilde_q_zero}
\end{align}


We shall first obtain the expected NPV of dividends (see the decomposition \eqref{X_r_decomposition}) killed upon exiting $[a,b]$.

\begin{theorem}[Periodic control of dividends] \label{prop_dividends} For $q \geq 0$, $a < 0 < b$, and $x  \leq b$, we have
	 \begin{align*} 
	 f(x,a,b) := \E_x\Big(\int_0^{\tau_b^+(r) \wedge\tau_a^- (r)}e^{-qt} \diff L_r(t)\Big)
		& = r \Big(  \overline{\overline{W}}^{(q+r)}(b)  \frac {I_a^{(q,r)}(x)} {I_a^{(q,r)}(b)} -\overline{\overline{W}}^{(q+r)}(x)   \Big).
	 \end{align*}
\end{theorem}


By taking  $a \downarrow -\infty$ and $b \uparrow \infty$ in Theorem \ref{prop_dividends}, we have the following.  

\begin{corollary}\label{cor_div}
(i) For $q \geq 0$, $b > 0$, and $x  \leq b$,  we have
	 \begin{align*}
\E_x\Big(\int_0^{\tau_b^+(r)}e^{-qt} \diff L_r(t)\Big)
		& = r \Big(   \overline{\overline{W}}^{(q+r)}(b)  \frac {I_{-\infty}^{(q,r)}(x)} {I_{-\infty}^{(q,r)}(b)} -\overline{\overline{W}}^{(q+r)}(x)   \Big),
	 \end{align*}
where 
	\begin{align} \label{I_infinity_conv}
	I_{-\infty}^{(q,r)}(x):= \lim_{a \downarrow -\infty} I_{a}^{(q,r)}(x)=Z^{(q+r)} (x, \Phi(q)) - r \overline{W}^{(q+r)}(x), \quad q \geq 0, \; x \in \R.
	\end{align}

(ii) For $q \geq 0$, $a < 0$, and $x  \in \R$, we have
	 	 	 \begin{align*}
\E_x\Big(\int_0^{\tau_a^-(r) }e^{-qt} \diff L_r(t)\Big)
& = r \Big(   \frac {I_a^{(q,r)}(x)} {\Phi(q+r)} \frac {W^{(q)}(-a)}  { Z^{(q)\prime}(-a, \Phi(q+r))} -\overline{\overline{W}}^{(q+r)}(x)  \Big),
	 \end{align*}
 where, by \eqref{Z_theta_derivative},
\begin{align*}
Z^{(q)\prime}(x, \Phi(q+r)) = \Phi(q+r)Z^{(q)} (x, \Phi(q+r)) - r W^{(q)}(x), \quad x \in \R.
\end{align*}
	 
(iii) Suppose $q > 0$ or $q = 0$ with $\psi'(0+) < 0$.  Then, for $x  \in \R$, 
	 	 \begin{align*}
\E_x\left(\int_0^{\infty}e^{-qt} \diff L_r(t)\right)
				& =   \frac  {\Phi(q+r) - \Phi(q)} {\Phi(q+r) \Phi(q)} I_{-\infty}^{(q,r)}(x) -r\overline{\overline{W}}^{(q+r)}(x).
			 \end{align*}
Otherwise, it is infinity for $x \in \R$.
	 \end{corollary}

 
 We shall now study the two-sided exit identities and their corollaries.
 
 \begin{theorem}[Up-crossing time] \label{proposition_upcrossing_time} For $q \geq 0$, $a < 0 < b$, and $x \leq b$, we have
 \begin{align*} 
 g(x,a,b) :=\mathbb{E}_x\left(e^{-q\tau_b^+(r)};\tau_a^-(r) > \tau_b^+(r)\right)
= \frac {I_a^{(q,r)}(x)}  {I_a^{(q,r)}(b)}.
\end{align*}
 \end{theorem}
 
 \begin{remark}
Fix $b > 0$ and $x < b$. By Lemma \ref{lemma_H_X} below, we see that $I_a^{(q,r)}(x)  - W^{(q+r)}(x)  I_a^{(q,r)}(b) / W^{(q+r)}(b)\xrightarrow{r \uparrow \infty} 1$.  Because $I_a^{(q,r)}(b)  \xrightarrow{r \uparrow \infty} \infty$ and by \eqref{W^{(q)}_limit}, \begin{align*}
\lim_{r \uparrow \infty}\frac {I_a^{(q,r)}(x)} {I_a^{(q,r)}(b)} = \lim_{r \uparrow \infty} \frac {W^{(q+r)}(x)} {W^{(q+r)}(b)} = 0.
\end{align*}
Hence, we see that $g(x, a,b)$ vanishes in the limit as $r \uparrow \infty$.
\end{remark}

By taking $a \downarrow -\infty$ in Theorem \ref{proposition_upcrossing_time}, we have the following.
 \begin{corollary} \label{corollary_g} (i) For $q \geq 0$, $b > 0$, and $x \leq b$, we have $\mathbb{E}_x (e^{-q\tau_b^+(r)}) =  {I_{-\infty}^{(q,r)}(x)} / {I_{-\infty}^{(q,r)}(b)}$ where $I_{-\infty}^{(q,r)}$ is given as in \eqref{I_infinity_conv}.
(ii) In particular, when $\psi'(0+) \geq 0$, then $\tau_b^+(r) < \infty$ $\p_x$-a.s.\ for any $x \in \R$.
 \end{corollary}


For $\theta \geq 0$, $q \geq 0$, $a < 0$, and $x \in \R$, let
\begin{align*}
\hat{J}_a^{(q,r)}(x,\theta ) &:= Z^{(q,r)}_a(x,\theta ) - \frac {Z^{(q)} (-a, \theta )} {W^{(q)}(-a)}   W^{(q,r)}_a(x) = \mathcal{M}_a^{(q,r)} \Big( Z^{(q)}(x,\theta ) - \frac {Z^{(q)} (-a, \theta )} {W^{(q)}(-a)}   W^{(q)}(x) \Big), 
\end{align*}
which satisfies
\begin{align}
\hat{J}_a^{(q,r)}(x,\theta )  &= J_a^{(q,r)}(x,\theta ) - Z^{(q)} (-a, \theta ) I_a^{(q,r)} (x),  \label{J_J_tilde_relation}\end{align}
and, by \eqref{I_J_zero},
\begin{align}
\hat{J}_a^{(q,r)}(0,\theta ) = 0. \label{hat_J_zero}
\end{align}

\begin{theorem}[Down-crossing time and overshoot]\label{proposition_laplace}
For $q \geq 0$, $a < 0 < b$, $\theta  \geq 0$, and $x \leq b$, we have
\begin{align} \label{h_a_b_expression}
\begin{split}
h(x,a,b,\theta ) &:=\mathbb{E}_x\left(e^{-q\tau_a^-(r)-\theta  [a-X_r (\tau_a^- (r))]};\tau_a^-(r) <\tau_b^+ (r)\right) \\ &= \hat{J}_a^{(q,r)}(x,\theta )  -\frac {I_a^{(q,r)}(x)}   {I_a^{(q,r)}(b)} \hat{J}_a^{(q,r)}(b,\theta ) 
= 
  J_a^{(q,r)}(x,\theta )  -\frac {I_a^{(q,r)}(x)}   {I_a^{(q,r)}(b)} J_a^{(q,r)}(b,\theta ).
  \end{split}
\end{align}


\end{theorem}

By taking $b \uparrow \infty$ in Theorem \ref{proposition_laplace}, we obtain the following.
\begin{corollary} \label{corollary_h_a}
(i)  For $q \geq 0$, $a < 0 $, $\theta  \geq 0$, and $x\in\R$, 
\begin{multline*}
\mathbb{E}_x\left(e^{-q\tau_a^-(r)-\theta  [a-X_r (\tau_a^-(r))]} \right) =  J_a^{(q,r)}(x,\theta )  -  {I_a^{(q,r)}(x)} \Big( \tilde{Z}^{(q,r)}(-a,\theta ) - \frac {r Z^{(q)}(-a, \theta)} {\Phi(q+r)} \Big)  \frac {W^{(q)}(-a) \Phi(q+r)} {
Z^{(q)\prime}  (-a, \Phi(q+r))
},\end{multline*}
where in particular
\begin{align*}
\mathbb{E}_x\left(e^{-q\tau_a^-(r)} \right)  
&=   J_a^{(q,r)}(x)  -   q{I_a^{(q,r)}(x)}  Z^{(q)}(-a,\Phi(q+r))  \frac {W^{(q)}(-a) } {
Z^{(q)\prime}  (-a, \Phi(q+r))}.
\end{align*}

(ii) For $a < 0$ and $x \in \R$, $\tau_a^-(r)< \infty$ $\p_x$-a.s. 
\end{corollary}

By taking $\theta \uparrow \infty$ in Theorem \ref{proposition_laplace} and Corollary \ref{corollary_h_a}, we have the following.
\begin{corollary} [Creeping]  \label{corollary_creeping} (i) For $q \geq 0$, $a < 0 < b$, and $x \leq b$, we have
\begin{align*}
w(x,a,b):=	\E_x\left(e^{-q\tau_a^-(r)}; X(\tau_a^-(r))=a, \tau_a^-(r)<\tau_b^+(r)\right)=   C_a^{(q,r)}(x)  - \frac {I_a^{(q,r)}(x)}   {I_a^{(q,r)}(b)}  C_a^{(q,r)}(b) 
\end{align*}
where (recall that $W^{(q)}$ is differentiable when $\sigma > 0$ as in Remark \ref{remark_smoothness_zero} (1))
\begin{align*}
 C_a^{(q,r)}(y) := \frac {\sigma^2} 2\left(\mathcal{M}_a^{(q,r)} W^{(q)\prime}(y) - r \overline{W}^{(q+r)} (y)  W^{(q)\prime}(-a)\right), \quad y \in \R.
\end{align*}
(ii)  For $q \geq 0$, $a < 0 $, and $x \in \R$, we have
\begin{align*}
	\E_x\Big(e^{-q\tau_a^-(r)}; X(\tau_a^-(r))=a \Big)=  C_a^{(q,r)}(x)  -  {I_a^{(q,r)}(x)} W^{(q)}(-a)  \frac{\sigma^2}{2}\Big[  \Phi(q+r)  - r \frac {W^{(q)\prime}(-a)  }  {
		Z^{(q)\prime}  (-a, \Phi(q+r))
	}  \Big] .
\end{align*}
\end{corollary}


In Theorem \ref{proposition_laplace},
by taking the derivative with respect to $\theta $ and taking $\theta\downarrow0$, we obtain the following.  This will later be used to compute the identities for capital injection in Proposition \ref{prop_capital_injection}.

\begin{corollary} \label{corollary_overshoot}
Suppose $\psi'(0+) > -\infty$. For $q \geq 0$, $a < 0< b$, and $x \leq b$, we have
\begin{align*}
j(x,a,b):= \E_x \Big( e^{-q \tau_a^-(r)} [a-X_r(\tau_a^-(r))]; \tau_a^- (r)< \tau_b^+(r) \Big)
&= \frac {I_a^{(q,r)}(x)}   {I_a^{(q,r)}(b)} K^{(q,r)}_a(b) - K^{(q,r)}_a(x)\end{align*}
with
\begin{align*}
 K^{(q,r)}_a(y) :=  l_a^{(q,r)}(y)  - r l^{(q)}(-a)
  \overline{W}^{(q+r)} (y), \quad y \in \R,
 \end{align*}
where $l_a^{(q,r)}(y) := \mathcal{M}_a^{(q,r)} l^{(q)}(y)$, $y \in \R$.
\end{corollary}

By taking $b \uparrow \infty$ in Corollary \ref{corollary_overshoot}, we have the following.
\begin{corollary} \label{corollary_overshoot_limit} 
	Suppose $\psi'(0+) > -\infty$. For $q \geq 0$, $a < 0$, and $x \in \R$, we have
\begin{align*}
\E_x \Big( e^{-q \tau_a^-(r)} [a-X_r(\tau_a^-(r))] \Big) = \frac { I_a^{(q,r)}(x)  W^{(q)}(-a)}  {
Z^{(q)\prime}  (-a, \Phi(q+r))
}  \Big(  \tilde{Z}^{(q,r)}(-a)- \psi'(0+) {Z^{(q)}(-a, \Phi(q+r))} \Big) - K^{(q,r)}_a(x).\end{align*}
\end{corollary}

\begin{remark}
Note that, for $q \geq 0$, $a < 0$, and $x \in \R$,
\begin{align} \label{convergence_r_zero}
\lim_{r \downarrow 0}I_a^{(q,r)} (x) = \frac {W^{(q)}(x-a)} {W^{(q)}(-a)}  \quad \textrm{and} \quad \lim_{r \downarrow 0} J_a^{(q,r)}(x,\theta ) &=Z^{(q)}(x-a,\theta ). 
\end{align}
Hence, as $r \downarrow 0$, we have the following.
\begin{enumerate}
\item By Theorem \ref{prop_dividends}, $f(x,a,b)$ vanishes in the limit.
\item By Theorems \ref{proposition_upcrossing_time} and \ref{proposition_laplace}, $g(x,a,b)$ and   $h(x,a,b,\theta )$ converge to the right hand sides of \eqref{laplace_in_terms_of_z}.  
\item 
By Corollary \ref{corollary_creeping} (i), $w(x,a,b)$ converges to the right hand side of \eqref{creeping_two_sided}.
\end{enumerate}
The convergence for the limiting cases $a = -\infty$ and/or $b = \infty$ hold in the same way. 
\end{remark}

\section{Main results for the cases with additional classical reflections} \label{section_results2}

In this section, we shall extend the results in Section \ref{section_results1} and obtain similar identities for the processes $\tilde{X}_r^b$, $Y_r^a$, and $\tilde{Y}_r^{a,b}$ as defined in Sections \ref{subsection_parisian_classical_above}, \ref{subsection_with_reflection_below}, and \ref{subsection_double_reflected_case}, respectively.  Again, the proofs for the corollaries are deferred to the appendix.

\subsection{Results for $\tilde{X}_r^b$}  We shall first study the process $\tilde{X}_r^b$ as constructed in Section \ref{subsection_parisian_classical_above}. 
 Let
\begin{align*} 
\tilde{\tau}_{a,b}^-(r) := \inf \{ t > 0: \tilde{X}^b_r(t)  < a\}, \quad a < 0 < b, \end{align*}
and $(I_a^{(q,r)})'(x+)$ be the right-hand derivative of \eqref{def_I} with respect to $x$ given by:
\begin{align*} 
(I_a^{(q,r)})' (x+) := \frac {(W^{(q,r)}_a)'(x+)} {W^{(q)}(-a)}  - r  W^{(q+r)} (x), \quad q \geq 0, \;a < 0, \; x \in \R.
\end{align*}

Recall the classical reflected process $\overline{Y}^b$ and $\tilde{\tau}_{0,b}^-$ as in \eqref{downcrossing_classical_reflected}.
We shall first compute the following.


\begin{lemma} \label{lemma_H}For $q \geq 0$ and $a < 0 < b$,
\begin{align*}
\E_b\left( e^{-q \mathbf{e}_r} ;\mathbf{e}_r < \tilde{\tau}_{0,b}^-  \right)  + \E_b \Big( e^{-(q+r) \tilde{\tau}_{0,b}^-} \frac {W^{(q)}(\overline{Y}^b(\tilde{\tau}_{0,b}^-)-a)}  {W^{(q)}(-a)}   \Big)
&= I_a^{(q,r)}(b)  - \frac{W^{(q+r)}(b)}{W^{(q+r) \prime}(b+)} (I_a^{(q,r)})'(b +).
\end{align*}
\end{lemma}
\begin{proof} 
We first note that, by \eqref{downcrossing_time_reflected},
\begin{align*}
\E_b\left( e^{-q \mathbf{e}_r}; \mathbf{e}_r < \tilde{\tau}_{0,b}^-  \right)=\frac{r}{r+q}\E_b\left(1-e^{-(q+r) \tilde{\tau}_{0,b}^-}\right)=r \Big(  \frac {(W^{(q+r)} (b))^2} {W^{(q+r)\prime} (b+)} - \overline{W}^{(q+r)} (b)  \Big).
\end{align*}
By summing this and \eqref{scale_function_overshoot_simplifying_reflected}, the result follows.
\end{proof}


In order to obtain the results for $\tilde{X}_r^b$, we  shall use the following observation and the strong Markov property.
\begin{remark} \label{remark_on_X_r_tilde}  
(i) For $0 \leq t < \tilde{\tau}_{0,b}^- \wedge \mathbf{e}_r$, $\tilde{X}_r^b(t) = \overline{Y}^b(t)$ and $\tilde{L}_{r, P}^b(t) = 0$.  (ii) For $0 \leq t \leq \tau_0^+$,  $\tilde{X}_r^b(t) = X(t)$ and $\tilde{L}_{r, P}^b(t) = \tilde{L}_{r,S}^b(t) = 0$.  (iii) For $0 \leq t \leq \tau_b^+(r)$, $\tilde{X}_r^b(t) = X_r(t)$.
\end{remark}



We shall first compute the expected NPV of the periodic part of dividends.  
	

	\begin{proposition}[Periodic part of dividends]\label{prop_f_tilde_p}
		For $q \geq 0$, $a<0<b$, and $x\leq b$, we have
		\begin{align*}
			\tilde{f}_P(x,a,b):= \E_x\Big(\int_0^{\tilde{\tau}_{a,b}^-(r)}e^{-qt}\diff \tilde{L}_{r,P}^b(t)\Big) =r \Big( \overline{W}^{(q+r)}(b)\frac {I_a^{(q,r)}(x)} {(I_a^{(q,r)})'(b +)} -\overline{\overline{W}}^{(q+r)}(x)\Big).
		\end{align*}
	\end{proposition}
	\begin{proof} 
By Remark \ref{remark_on_X_r_tilde} (i) and the strong Markov property, we can write
		\begin{align} \label{f_tilde_p_at_b}
		\begin{split}
			\tilde{f}_P(b,a,b)&=\E_b\left(e^{-q \tilde{\tau}_{0,b}^-}\tilde{f}_P(\overline{Y}^b(\tilde{\tau}_{0,b}^-),a,b); \tilde{\tau}_{0,b}^- <\mathbf{e}_r\right)  +\E_b\left(e^{-q\mathbf{e}_r}[\overline{Y}^b(\mathbf{e}_r) + \tilde{f}_P(0,a,b)] ;\mathbf{e}_r< \tilde{\tau}_{0,b}^- \right).
\end{split}
\end{align}
For $x \leq 0$, by Remark \ref{remark_on_X_r_tilde} (ii) and the strong Markov property, $\tilde{f}_P(x,a,b)=\E_x (e^{- q\tau_0^+};\tau_0^+<\tau_a^- )\tilde{f}_P(0,a,b)$.
This together with  \eqref{laplace_in_terms_of_z} gives
\begin{align} \label{f_tilde_exp_W}
		\E_b\left(e^{-q \tilde{\tau}_{0,b}^-}\tilde{f}_P(\overline{Y}^b(\tilde{\tau}_{0,b}^-),a,b); \tilde{\tau}_{0,b}^- <\mathbf{e}_r\right) &=\frac{\tilde{f}_P(0,a,b)}{W^{(q)}(-a)}\E_b\left(e^{-q \tilde{\tau}_{0,b}^-}W^{(q)}(\overline{Y}^b(\tilde{\tau}_{0,b}^-)-a); \tilde{\tau}_{0,b}^- <\mathbf{e}_r\right).
\end{align}
On the other hand, by the resolvent given in Theorem 1 (ii) of \cite{P2004},
		\begin{align} \label{Y_before_downcrossing}
		\begin{split}
			&\E_b\left(e^{-q\mathbf{e}_r}\overline{Y}^b(\mathbf{e}_r);\mathbf{e}_r<\tilde{\tau}_{0,b}^-\right)=r\E_b\Big(\int_0^{\tilde{\tau}_{0,b}^-}e^{-(q+r)s}\overline{Y}^b(s) \diff s\Big)\\
			&=r\int_0^b(b-y)\left(W^{(q+r)}(b)\frac{W^{(q+r)\prime}(y)}{W^{(q+r)\prime}(b+)}-W^{(q+r)}(y)\right)\diff y+b r W^{(q+r)}(b)\frac{W^{(q+r)}(0)}{W^{(q+r)\prime}(b +)}\\
			&=r\left(\frac{W^{(q+r)}(b)}{W^{(q+r)\prime}(b+)}\overline{W}^{(q+r)}(b)-\overline{\overline{W}}^{(q+r)}(b)\right).
			\end{split}
		\end{align}
	Substituting \eqref{f_tilde_exp_W} and \eqref{Y_before_downcrossing} in \eqref{f_tilde_p_at_b}, and applying Lemma \ref{lemma_H},
		\begin{align*}
			\tilde{f}_P(b,a,b)
			&= \Big( I_a^{(q,r)}(b)  - \frac{W^{(q+r)}(b)}{W^{(q+r) \prime}(b+)} (I_a^{(q,r)})'(b +) \Big) \tilde{f}_P(0,a,b)  
			+r\Big(\frac{W^{(q+r)}(b)}{W^{(q+r)\prime}(b+)}\overline{W}^{(q+r)}(b)-\overline{\overline{W}}^{(q+r)}(b)\Big).
		\end{align*}

Now  by Remark \ref{remark_on_X_r_tilde} (iii), the strong Markov property, and Theorems \ref{prop_dividends} and \ref{proposition_upcrossing_time}, for all $x \leq b$,
		\begin{align} \label{f_tilde_P_recurs}
		\begin{split}
			&\tilde{f}_P(x,a,b)=
			f(x,a,b)
			+g(x,a,b) \tilde{f}_P(b,a,b) =-r\overline{\overline{W}}^{(q+r)}(x)   +   \frac {I_a^{(q,r)}(x)} {I_a^{(q,r)}(b)} \Big( r\overline{\overline{W}}^{(q+r)}(b)  +\tilde{f}_P(b,a,b) \Big) \\
						&=-r\overline{\overline{W}}^{(q+r)}(x)   +\frac {I_a^{(q,r)}(x)} {I_a^{(q,r)}(b)} \Big[ \Big( I_a^{(q,r)}(b)  - \frac{W^{(q+r)}(b)}{W^{(q+r) \prime}(b+)} (I_a^{(q,r)})'(b +) \Big)  \tilde{f}_P(0,a,b)  +r\frac{W^{(q+r)}(b)}{W^{(q+r)\prime}(b +)}\overline{W}^{(q+r)}(b) \Big].
			\end{split}
		\end{align}
		Setting $x=0$ and solving for $\tilde{f}_P(0,a,b)$ (using \eqref{I_J_zero}), we have 
$\tilde{f}_P(0,a,b)=r {\overline{W}^{(q+r)}(b)}/{(I_a^{(q,r)})'(b+)}$.
		Substituting this back in \eqref{f_tilde_P_recurs}, the proof is complete. 
%

	\end{proof}
	By taking $a \downarrow -\infty$ in Proposition \ref{prop_f_tilde_p}, we have the following.
\begin{corollary} \label{corollary_L_tilde_r_P}(i) For $q > 0$ or $q = 0$ with $\psi'(0+) < 0$, we have, for $b > 0$ and $x \leq b$,
	\begin{align*}
		\E_x\left(\int_0^\infty e^{-qt} \diff \tilde{L}_{r,P}^b(t)\right) = r \Big( \overline{W}^{(q+r)}(b) \frac {I_{-\infty}^{(q,r)}(x)} {(I_{-\infty}^{(q,r)})'(b)} -\overline{\overline{W}}^{(q+r)}(x) \Big), \end{align*}
		where $(I_{-\infty}^{(q,r)})'$ is the derivative of $I_{-\infty}^{(q,r)}$ of \eqref{I_infinity_conv}  given by
\begin{align*}
(I_{-\infty}^{(q,r)})'(x) =Z^{(q+r)\prime}(x,\Phi(q))-rW^{(q+r)}(x) = \Phi(q) Z^{(q+r)}(x, \Phi(q)), \quad q \geq 0, \; x \in \R.
\end{align*}
(ii) If $q = 0$ with $\psi'(0+) \geq 0$, it becomes infinity.
\end{corollary}

Now consider the singular part of dividends. 
	 
\begin{proposition}[Singular part of dividends] \label{prop_f_tilde_S}For $q \geq 0$, $a < 0 < b$, and $x \leq b$, we have 
\begin{align*}
\tilde{f}_S(x,a,b) := \E_x\Big(\int_{[0,\tilde{\tau}_{a,b}^-(r)]}e^{-qt} \diff \tilde{L}_{r,S}^{b}(t)\Big) = \frac {I_a^{(q,r)}(x)} {(I_a^{(q,r)})'(b +)}.
\end{align*}
\end{proposition}
\begin{proof}
By Remark \ref{remark_on_X_r_tilde} (i) and the strong Markov property,
\begin{multline*}
\tilde{f}_S(b,a,b) = \E_b\Big(\int_{[0,\tilde{\tau}_{0,b}^- \wedge \mathbf{e}_{r}]}e^{-qt} \diff L^b(t)\Big)   \\
+ \E_b\left( e^{-q \mathbf{e}_{r}}; \mathbf{e}_{r} < \tilde{\tau}_{0,b}^- \right)  \tilde{f}_S(0,a,b) 
+ \E_b \Big( e^{-q \tilde{\tau}_{0,b}^-} \tilde{f}_S(\overline{Y}^b(\tilde{\tau}_{0,b}^-),a,b); \tilde{\tau}_{0,b}^- < \mathbf{e}_{r}\Big).
\end{multline*}
By \eqref{dividend_classical_barrier} and the computation similar to \eqref{f_tilde_exp_W} (thanks to Remark \ref{remark_on_X_r_tilde} (ii)),
\begin{align} \label{f_S_tilde_recurs}
\tilde{f}_S(b,a,b) 
&= \frac {W^{(q+r)}(b)} {W^{(q+r)\prime}(b+)}  +  \tilde{f}_S(0,a,b) \Big( I_a^{(q,r)}(b)  - \frac{W^{(q+r)}(b)}{W^{(q+r) \prime}(b+)} (I_a^{(q,r)})'(b +) \Big) . 
\end{align}
For $x \leq b$, because Remark \ref{remark_on_X_r_tilde} (iii) and the strong Markov property give $\tilde{f}_S(x,a,b) = g(x,a,b) \tilde{f}_S(b,a,b)$,  Theorem \ref{proposition_upcrossing_time} and \eqref{f_S_tilde_recurs} give
\begin{align}\label{f_S_x}
&\tilde{f}_S(x,a,b) =  \frac {I_a^{(q,r)}(x)} {I_a^{(q,r)}(b)} \Big[ \frac {W^{(q+r)}(b)} {W^{(q+r)\prime}(b+)} +  \tilde{f}_S(0,a,b) \Big( I_a^{(q,r)}(b)  -  \frac{W^{(q+r)}(b)}{W^{(q+r) \prime}(b+)  } (I_a^{(q,r)})'(b+) \Big) 
\Big].
\end{align}
Setting $x = 0$ and solving for $\tilde{f}_S(0,a,b)$ (using \eqref{I_J_zero}), 
we have $\tilde{f}_S(0,a,b)  =  [(I^{(q,r)}_a)'(b +)]^{-1}$.
Substituting this in \eqref{f_S_x}, we have the result.
\end{proof}
By taking $a \downarrow -\infty$ in Proposition \ref{prop_f_tilde_S}, we have the following.
\begin{corollary} \label{corollary_f_tilde_S} Fix $b > 0$ and $x \leq b$. (i) For $q > 0$  or $q = 0$ with $\psi'(0+) < 0$, we have 
		$\E_x (\int_0^\infty e^{-qt}\diff \tilde{L}_{r,S}^b(t) ) =   {I_{-\infty}^{(q,r)}(x)}   /{(I_{-\infty}^{(q,r)})'(b +)}$. 
(ii) If $q = 0$ with $\psi'(0+) \geq 0$, it becomes infinity.
\end{corollary}
	\begin{proposition}[Down-crossing time and overshoot] \label{prop_g_reflection_above} 
	Fix $a < 0 < b$ and $x \leq b$.
		(i)  For $q \geq 0$ and $\theta  \geq 0$, 
		\begin{align} \label{h_tilde_identity}
\tilde{h} (x,a,b, \theta ) :=\E_x \big(e^{-q \tilde{\tau}_{a,b}^-(r) - \theta  [a-\tilde{X}_r(\tilde{\tau}_{a,b}^-(r))]} \big)= J_a^{(q,r) }(x, \theta ) -  (J_a^{(q,r)})'(b, \theta ) \frac{I_a^{(q,r)}(x)}{(I_a^{(q,r)})'(b +)}.
		\end{align}
		(ii)  We have $\tilde{\tau}_{a,b}^-(r) < \infty$, $\p_x$-a.s. 
	\end{proposition}
\begin{proof}
(i) By Remark \ref{remark_on_X_r_tilde} (i) and the strong Markov property, we can write
\begin{align*}
\tilde{h}  (b,a,b, \theta )=\E_b\left(e^{-q \tilde{\tau}_{0,b}^-}\tilde{h} (\overline{Y}^b(\tilde{\tau}_{0,b}^-),a,b, \theta ); \tilde{\tau}_{0,b}^- <\mathbf{e}_r\right)+\E_b\left(e^{-q\mathbf{e}_r}; \mathbf{e}_r < \tilde{\tau}_{0,b}^- \right)\tilde{h}  (0,a,b, \theta ).
\end{align*}
For $x \leq 0$, by Remark \ref{remark_on_X_r_tilde} (ii), the strong Markov property, and \eqref{laplace_in_terms_of_z},
\begin{align*}
\tilde{h}  (x,a,b,\theta ) &= \E_x \left( e^{-q \tau_a^- - \theta  [a-X(\tau_a^-)]}; \tau_0^+ > \tau_a^- \right) + \E_x \left( e^{-q \tau_0^+}; \tau_0^+ < \tau_a^- \right) \tilde{h} (0,a,b,\theta ) \\
&=Z^{(q)}(x-a,\theta ) -  Z^{(q)}(-a,\theta ) \frac {W^{(q)}(x-a)}  {W^{(q)}(-a)}
+\tilde{h} (0,a,b,\theta )\frac{W^{(q)}(x-a)}{W^{(q)}(-a)},
\end{align*}
and hence, together with \eqref{scale_function_overshoot_simplifying_reflected} and  Lemmas \ref{lemma_simplifying_formula_measure_changed} (ii) and \ref{lemma_H},
\begin{align*}
\tilde{h}  (b,a,b, \theta )
&= \E_b \left( e^{-(q+r) \tilde{\tau}_{0,b}^-} \Big( Z^{(q)}(\overline{Y}^b(\tilde{\tau}_{0,b}^-)-a,\theta ) -  Z^{(q)}(-a,\theta ) \frac {W^{(q)}(\overline{Y}^b(\tilde{\tau}_{0,b}^-)-a)}  {W^{(q)}(-a)} \Big) \right) \\
&+\tilde{h}  (0,a,b, \theta ) \Big[ \E_b\left( e^{-q \mathbf{e}_r} ;\mathbf{e}_r < \tilde{\tau}_{0,b}^-  \right)  + \frac {1} {W^{(q)}(-a)} \E_b \Big( e^{-(q+r) \tilde{\tau}_{0,b}^-} W^{(q)}(\overline{Y}^b(\tilde{\tau}_{0,b}^-)-a)  \Big) \Big] \\
&=  Z_a^{(q,r)}(b,\theta ) - \frac {W^{(q+r)}(b)} {W^{(q+r)\prime} (b+)} (Z_a^{(q,r)})'(b,\theta ) -\frac {Z^{(q)} (-a,\theta) } {W^{(q)}(-a)} \Big(W^{(q,r)}_a (b)-\frac{W^{(q+r)}(b)}{W^{(q+r) \prime}(b+)} (W^{(q,r)}_a)' (b) \Big) \\ &+ \tilde{h} (0,a,b, \theta ) \Big( I_a^{(q,r)}(b)  - \frac{W^{(q+r)}(b)}{W^{(q+r) \prime}(b+)} (I_a^{(q,r)})'(b+) \Big) \\
&= \hat{J}_a^{(q,r)}(b,\theta ) - \frac {W^{(q+r)}(b)} {W^{(q+r)\prime} (b+)} (\hat{J}_a^{(q,r)})'(b,\theta )  + \tilde{h} (0,a,b, \theta ) \Big( I_a^{(q,r)}(b)  - \frac{W^{(q+r)}(b)}{W^{(q+r) \prime}(b+)} (I_a^{(q,r)})'(b+) \Big).
\end{align*}

On the other hand, by Remark \ref{remark_on_X_r_tilde} (iii), the strong Markov property, and Theorems \ref{proposition_upcrossing_time} and \ref{proposition_laplace} (ii), we have that, for all $x \leq b$, 
\begin{align} \label{h_recursive_v}
\begin{split}
&\tilde{h} (x,a,b, \theta )=h(x,a,b,\theta ) + g(x,a,b) \tilde{h} (b,a,b, \theta ) \\
&=    \hat{J}_a^{(q,r)}(x,\theta )  -\frac {I_a^{(q,r)}(x)}   {I_a^{(q,r)}(b)}  \hat{J}_a^{(q,r)}(b,\theta )  
 + \frac{I_a^{(q,r)}(x)}{I_a^{(q,r)}(b)} \Big[  \hat{J}_a^{(q,r)}(b,\theta ) - \frac {W^{(q+r)}(b)} {W^{(q+r)\prime} (b+)} (\hat{J}_a^{(q,r)})'(b,\theta )  \\
&+ \tilde{h} (0,a,b, \theta ) \Big( I_a^{(q,r)}(b)  - \frac{W^{(q+r)}(b)}{W^{(q+r) \prime}(b+)} (I_a^{(q,r)})' (b+) \Big)  \Big] \\
&=   \hat{J}_a^{(q,r)}(x,\theta)   
 + \frac{I_a^{(q,r)}(x)}{I_a^{(q,r)}(b)} \Big[   -\frac {W^{(q+r)}(b)} {W^{(q+r)\prime} (b+)} (\hat{J}_a^{(q,r)})'(b,\theta)   + \tilde{h} (0,a,b, \theta ) \Big( I_a^{(q,r)}(b)  - \frac{W^{(q+r)}(b)}{W^{(q+r) \prime}(b+)} (I_a^{(q,r)})' (b+) \Big)  \Big].
\end{split}
\end{align}
Setting $x = 0$, and solving for $\tilde{h} (0,a,b, \theta )$ (using \eqref{I_J_zero} and \eqref{hat_J_zero}), 
%
%
$\tilde{h} (0,a,b, \theta ) = - (\hat{J}_a^{(q,r)})' (b,\theta ) / (I_a^{(q,r)})' (b+)$.
Substituting this back in \eqref{h_recursive_v}, we have
		\begin{align*} 
\tilde{h} (x,a,b, \theta )= \hat{J}_a^{(q,r) }(x, \theta ) -  (\hat{J}_a^{(q,r)})'(b, \theta ) \frac{I_a^{(q,r)}(x)}{(I_a^{(q,r)})'(b+)}.
		\end{align*}
Using \eqref{J_J_tilde_relation}, it equals the right-hand side of \eqref{h_tilde_identity}.

(ii) 
In view of (i), it is immediate by \eqref{J_tilde_q_zero} by setting $q = \theta = 0$.
\end{proof}

%
%

Similarly to Corollary \ref{corollary_overshoot}, we obtain the following by Proposition  \ref{prop_g_reflection_above}.
\begin{corollary} \label{cor_overshoot_reflected}
Suppose $\psi'(0+) > -\infty$. For $q \geq 0$, $a < 0 < b$, and $x \leq b$, we have that
\begin{align*}
\tilde{j}(x,a,b):=\E_x\left(e^{-q\tilde{\tau}_{a,b}^-(r)}[a-\tilde{X}_r^b(\tilde{\tau}_{a,b}^-(r))]\right)=\frac{I_a^{(q,r)}(x)}{(I_a^{(q,r)})'(b +)}(K_a^{(q,r)})'(b)-K_a^{(q,r)}(x).
\end{align*}
\end{corollary}

\begin{remark} Recall \eqref{convergence_r_zero}.  As $r \downarrow 0$, we have the following.
\begin{enumerate}
\item By Proposition \ref{prop_f_tilde_p}, $\tilde{f}_P(x,a,b)$ vanishes in the limit.
\item By Proposition \ref{prop_f_tilde_S}, $\tilde{f}_S(x,a,b)$ converges to the right hand side of \eqref{dividend_classical_barrier}.
\item By Proposition \ref{prop_g_reflection_above}, $\tilde{h}(x,a,b,\theta )$ converges to 
\begin{align*}
\E_x \big(e^{-q \tilde{\tau}_{a,b}^- - \theta  [a-\overline{Y}^b(\tilde{\tau}_{a,b}^-)]} \big) = Z^{(q) }(x-a, \theta ) -  Z^{(q)\prime}(b-a, \theta ) \frac{W^{(q)}(x-a)}{W^{(q)\prime}((b-a) +)},
\end{align*}
which is given in Theorem 1 of \cite{AKP2004}.
\item By Corollary \ref{cor_overshoot_reflected}, $\tilde{j}(x,a,b)$ converges to \begin{align*}
\E_x\left(e^{-q\tilde{\tau}_{a,b}^-}[a-\overline{Y}^b(\tilde{\tau}_{a,b}^-)]\right) = \frac{W^{(q)}(x-a)}{W^{(q)\prime}((b-a) +)}l^{(q)\prime}(b-a)-l^{(q)}(x-a),
\end{align*}
which is given in (3.16) of \cite{APP2007}.
\end{enumerate}
The convergence for the limiting case $a = -\infty$ holds in the same way. 
\end{remark}

\subsection{Results for $Y_r^a$}
We shall now study the process $Y_r^a$ as defined in Section \ref{subsection_with_reflection_below}. 
We let
\begin{align*}
\eta_{a,b}^+(r) := \inf \{ t > 0: Y_r^a(t)  > b \}, \quad a < 0 < b. \end{align*}

\begin{remark} \label{remark_on_Y_r}
Recall the classical reflected process $\underline{Y}^a = X + R^a$ and $\eta_{a,0}^+$ as in \eqref{eta_a_b_plus}.   
(i) For $0 \leq t \leq \eta_{a,0}^+$, we have $Y_r^a(t) = \underline{Y}^a(t)$ and $R_r^a(t) = R^a(t)$.  (ii) For $0 \leq t < \tau_a^-(r)$, we have $Y_r^a(t) = X_r(t)$. 
\end{remark}

	 


	 \begin{proposition}[Periodic part of dividends]   \label{prop_f_hat} For $q \geq 0$, $a < 0 < b$, and $x \leq b$,
\begin{align*}
\hat{f}(x,a,b) :=\mathbb{E}_x\Big(\int_0^{\eta_{a,b}^+(r)}e^{-qt} \diff L_r^a(t)\Big)= r \Big( \overline{\overline{W}}^{(q+r)}(b)  \frac {J_a^{(q,r)}(x)} {  {J_a^{(q,r)}(b)}}  -\overline{\overline{W}}^{(q+r)}(x)  \Big).
\end{align*}
	 \end{proposition}
\begin{proof}

By an application of Remark \ref{remark_on_Y_r} (i), \eqref{upcrossing_time_reflected}, and the strong Markov property,
\[ 
\hat{f}(a,a,b)=  \E_a ( e^{- q \eta^+_{a,0}} ) \hat{f}(0,a,b)  = {\hat{f}(0,a,b)}/ {Z^{(q)}(-a)}.
\]
By this, Remark \ref{remark_on_Y_r} (ii), and the strong Markov property, together with Theorems \ref{prop_dividends} and \ref{proposition_laplace}, we have for $x \leq b$
\begin{align}\label{cl1}
\begin{split}
\hat{f}(x,a,b) &= f(x,a,b) +h(x,a,b,0) \hat{f}(a,a,b) \\
&=    r \Big( \overline{\overline{W}}^{(q+r)}(b)  \frac {I_a^{(q,r)}(x)} {I_a^{(q,r)}(b)} -\overline{\overline{W}}^{(q+r)}(x) \Big)    +\Big( J_a^{(q,r)}(x)
- J_a^{(q,r)}(b)
\frac{I^{(q,r)}_a(x)}{I^{(q,r)}_a(b)}\Big)  \frac {\hat{f}(0,a,b)}  {Z^{(q)}(-a)}.
\end{split}
\end{align}
%
%
%

Setting $x = 0$ and solving for $\hat{f}(0,a,b)$ (using \eqref{I_J_zero}), we get
$\hat{f}(0,a,b)= r\overline{\overline{W}}^{(q+r)}(b) Z^{(q)}(-a) / {J_a^{(q,r)}(b)}$.
Substituting this in \eqref{cl1}, we have the claim.
\end{proof}
By taking $b \uparrow \infty$ in Proposition \ref{prop_f_hat}, we have the following.
	 \begin{corollary}\label{corollary_dividends_limit} 
	 Fix $a < 0$ and $x \in \R$.
	(i) For $q > 0$,  we have
\begin{align*}
\mathbb{E}_x\left(\int_0^\infty e^{-qt} \diff L_r^a(t)\right)= r \Big(\frac{1}{q\Phi(q+r)}\frac{J_a^{(q,r)}(x)}{Z^{(q)}(-a,\Phi(q+r))}  -\overline{\overline{W}}^{(q+r)}(x)\Big).
\end{align*}
(ii) For $q = 0$, it becomes infinity.
	 \end{corollary}


For $q \geq 0$ and $a < 0$, let
\begin{align*} 
H_a^{(q,r)}(y) := l^{(q,r)}_a(y)   - \frac {l^{(q)}(-a)} {Z^{(q)}(-a)} Z_a^{(q,r)}(y) = K^{(q,r)}_a(y) -
		\frac {J_a^{(q,r)}(y)}  {Z^{(q)}(-a)} l^{(q)}(-a), \quad y \in \R.
\end{align*}
In particular, 
\begin{align} \label{H_zero}
H_a^{(q,r)}(0) = 0.
\end{align}

\begin{proposition}[Capital injections] \label{prop_capital_injection} For $q \geq 0$, $a < 0 < b$, and $x \leq b$,
\begin{align*}
	 \begin{split}
	 \hat{j}(x,a,b) &:=\mathbb{E}_x\Big(\int_{[0,\eta_{a,b}^+(r)]}e^{-qt} \diff R_r^a(t)\Big) = 
	 H_a^{(q,r)}(b) \frac {J_a^{(q,r)} (x) } {J_a^{(q,r)} (b) }  -H_a^{(q,r)}(x).
	 \end{split}
	 \end{align*}


	 \end{proposition}
\begin{proof}

First, by Remark \ref{remark_on_Y_r} (i), \eqref{upcrossing_time_reflected}, \eqref{overshoot_classical_expectation}, and an application of the strong Markov property, 
	 	 \begin{align}
		 \label{j_hat_recursion}
	 \hat{j}(a,a,b)&= \E_{a}\Big(\int_{[0,\eta^+_{a,0}]}e^{-qt}\diff R^a (t)\Big) + \E_a (e^{-q \eta^+_{a,0}}) \hat{j}(0,a,b)  
%
%
%
=\frac{l^{(q)}(-a)+\hat{j}(0,a,b)}{Z^{(q)}(-a)}.
\end{align}

This, together with Remark \ref{remark_on_Y_r} (ii), Corollary \ref{corollary_overshoot}, Theorem \ref{proposition_laplace}, and  the strong Markov property, gives, for $x \leq b$,
\begin{align}\label{ci_2}
\begin{split}
\hat{j}(x,a,b)&= 
j(x,a,b) +
h(x,a,b,0) \hat{j}(a,a,b) \\
&=\frac {I_a^{(q,r)}(x)}   {I_a^{(q,r)}(b)} K^{(q,r)}_a(b) - K^{(q,r)}_a(x) +
\Big[ J_a^{(q,r)}(x) -\frac {I_a^{(q,r)}(x)}   {I_a^{(q,r)}(b)}  J_a^{(q,r)}(b)   \Big]  \frac{l^{(q)}(-a) +\hat{j}(0,a,b)}{Z^{(q)}(-a)} \\
&=\frac {I_a^{(q,r)}(x)}   {I_a^{(q,r)}(b)} \Big( K^{(q,r)}_a(b) -
\frac {J_a^{(q,r)}(b)}  {Z^{(q)}(-a)} \Big[l^{(q)}(-a)  +{\hat{j}(0,a,b)}  \Big] \Big)  \\ &- K^{(q,r)}_a(x) +   \frac {J_a^{(q,r)}(x)} {Z^{(q)}(-a)} \Big[ l^{(q)}(-a)  +{\hat{j}(0,a,b)}  \Big] \\
&=\frac {I_a^{(q,r)}(x)}   {I_a^{(q,r)}(b)} \Big( H^{(q,r)}_a(b) -
\frac {J_a^{(q,r)}(b)}  {Z^{(q)}(-a)}{\hat{j}(0,a,b)}  \Big) - H^{(q,r)}_a(x) +   \frac {J_a^{(q,r)}(x)} {Z^{(q)}(-a)} {\hat{j}(0,a,b)}.
\end{split}
	\end{align}
%
%
%
Setting $x = 0$ and solving for $\hat{j}(0,a,b)$ (using \eqref{I_J_zero} and \eqref{H_zero}),  $\hat{j}(0,a,b) = H_a^{(q,r)}(b) Z^{(q)}(-a) / J_a^{(q,r)}(b)$.

Substituting this back in \eqref{ci_2}, we have the claim.
\end{proof}

By taking $b \uparrow \infty$ in Corollary \ref{prop_capital_injection}, we have the following.

\begin{corollary} \label{corollary_capital_injection_limit} 
For $q > 0$, $a < 0$, and $x \in \R$, we have 
\begin{align*}
\E_x\left(\int_{[0, \infty)}e^{-qt} \diff R_r^a(t)\right) &= \left(\frac{rZ^{(q)} (-a) }{q \Phi (q + r)Z^{(q)}(-a , \Phi (q + r) )} + \frac 1 {\Phi(q+r)}\right) \Big( Z_{a}^{(q, r)} (x) - r Z^{(q)}(-a) \overline{W}^{(q + r)} (x) \Big) \\
	&+ r \overline{Z}^{(q)} (-a)   \overline{W}^{(q + r)} (x)  - \Big( \overline{Z}^{(q,r)}_a (x) +\frac{\psi^\prime(0+)}{ q}   \Big).
	\end{align*}
\end{corollary}
\begin{proposition}[Up-crossing time]\label{ref_up_crossing_time}
Fix $a < 0 < b$.
(i) For $q > 0$ and $x \leq b$, we have 
\begin{align*}
\hat{g}(x,a,b) :=\mathbb{E}_x\left( e^{- q \eta_{a,b}^+(r)}\right)= \frac {J_a^{(q,r)}(x)} {J_a^{(q,r)}(b)}.
\end{align*} 
(ii) For all $x \in \R$, we have $\eta_{a,b}^+(r) < \infty$, $\p_x$-a.s.
	 \end{proposition}
\begin{proof}
(i) 
 By Remark \ref{remark_on_Y_r} (i) and the strong Markov property, together with \eqref{upcrossing_time_reflected}, 
  \[
  \hat{g}(a,a,b)=\E_{a}(e^{-q \eta^+_{a,0}} )\hat{g}(0,a,b) = \hat{g}(0,a,b) / {Z^{(q)}(-a)}.\]
By this, Remark \ref{remark_on_Y_r} (ii), and the strong Markov property, together with Theorems \ref{proposition_upcrossing_time} and \ref{proposition_laplace}, \begin{align}
\begin{split}
\hat{g}(x,a,b)
&= \frac {I_a^{(q,r)}(x)} {I_a^{(q,r)}(b)} +\Big( J_a^{(q,r)}(x)
- J_a^{(q,r)}(b)
\frac{I^{(q,r)}_a(x)}{I^{(q,r)}_a(b)}\Big) \frac {\hat{g}(0,a,b)} {Z^{(q)}(-a)}. \label{g_hat_recurs}
\end{split}
\end{align}
Setting $x = 0$ and by   \eqref{I_J_zero}, we obtain $\hat{g}(0,a,b) = Z^{(q)}(-a) /J_a^{(q,r)} (b)$.  Substituting this in \eqref{g_hat_recurs}, we have the claim.
(ii) This is immediate by setting $q = 0$ in (i) by \eqref{J_tilde_q_zero}.
\end{proof}

\begin{remark} Recall \eqref{convergence_r_zero}.  As $r \downarrow 0$, we have the following.
	\begin{enumerate}
		\item By Proposition \ref{prop_f_hat}, $\hat{f}(x,a,b)$ vanishes in the limit.
		\item By Proposition \ref{prop_capital_injection}, $\hat{j}(x,a,b)$ converges to the right hand side of \eqref{overshoot_classical_expectation}.
		\item By Proposition \ref{ref_up_crossing_time}, $\hat{g}(x,a,b)$ converges to the right hand side of \eqref{upcrossing_time_reflected}.
	\end{enumerate}
	The convergence for the limiting case $b = \infty$ holds in the same way. 
\end{remark}

\subsection{Results for $\tilde{Y}_r^{a,b}$}  We conclude this section with the identities for the process $\tilde{Y}_r^{a,b}$ as constructed in Section \ref{subsection_double_reflected_case}.  We use the derivative of $J_a^{(q,r)}$ as in \eqref{def_I}:
\begin{align*}
(J_a^{(q,r)})'(y)=(Z^{(q,r)}_a)'(y)-r Z^{(q)}(-a) W^{(q+r)}(y), \quad q \geq 0, \; a < 0, \; y \in \R. 
\end{align*}
We shall use the following observation and the strong Markov property.
\begin{remark} \label{remark_on_Y_r_tilde}
(i) For $0 \leq t \leq \eta_{a,0}^+$,  we have $\tilde{Y}_r^{a,b}(t) = \underline{Y}^a(t)$, $\tilde{L}_{r,P}^{a,b}(t) = \tilde{L}_{r,S}^{a,b}(t)= 0$, and $\tilde{R}_r^{a,b}(t) = R^a(t)$.  (ii) For all $0 \leq t < \tilde{\tau}_{a,b}^-(r)$, we have  $\tilde{Y}_r^{a,b}(t) = \tilde{X}_r^b(t)$, $\tilde{L}_{r, P}^{a,b}(t) = \tilde{L}_{r, P}^{b}(t)$, and $ \tilde{L}_{r,S}^{a,b}(t) = \tilde{L}_{r,S}^{b}(t)$.
\end{remark}


\begin{proposition}[Periodic part of dividends]\label{double_reflected_periodic} 
	Fix $a<0<b$ and $x \leq b$.
	(i) For $q > 0$, we have
	\begin{align*}
		\check{f}_P(x,a,b) :=\E_x\left(\int_0^\infty e^{-qt} \diff \tilde{L}_{r,P}^{a,b}(t)\right)=r \Big( \overline{W}^{(q+r)}(b)\frac {J^{(q,r)}_a(x)} {(J^{(q,r)}_a)'(b)} - \overline{\overline{W}}^{(q+r)}(x) \Big).
	\end{align*}
(ii) If $q =0$, it becomes infinity.
\end{proposition}
\begin{proof}
	By Remark \ref{remark_on_Y_r_tilde} (i),  \eqref{upcrossing_time_reflected}, and the strong Markov property, $\check{f}_P(a,a,b) =\E_a (e^{-q\eta_{a,0}^+} ) \check{f}_P(0,a,b) = \check{f}_P(0,a,b)  / {Z^{(q)}(-a)}$.
By this, Remark \ref{remark_on_Y_r_tilde} (ii), and the strong Markov property, together with Propositions \ref{prop_f_tilde_p} and \ref{prop_g_reflection_above}, for $x \leq b$,
	\begin{align} \label{per_triple_ref}
	\begin{split}
		\check{f}_P(x,a,b)  &=
		\tilde{f}_P(x,a,b)
		+ \tilde{h}(x,a,b,0)
		\check{f}_P(a,a,b)  \\
		&= r \Big(  \overline{W}^{(q+r)}(b)\frac {I^{(q,r)}_a(x)} {(I^{(q,r)}_a)'(b+)} -\overline{\overline{W}}^{(q+r)}(x)\Big) +\Big( J_a^{(q,r)}(x)
		-   (J_a^{(q,r)})'(b)
		\frac{I_a^{(q,r)}(x)}{(I_a^{(q,r)})'(b+)}\Big) \frac{\check{f}_P(0,a,b) }{Z^{(q)}(-a)}.
		\end{split}
	\end{align}
Now taking $x=0$ and by \eqref{I_J_zero}, we get
	$\check{f}_P(0,a,b) = r\overline{W}^{(q+r)}(b) Z^{(q)}(-a) / {(J_a^{(q,r)})'(b)}$.
	Substituting this in \eqref{per_triple_ref}, we have the claim. 
\end{proof}

\begin{proposition}[Singular part of dividends]\label{double_reflected_singular} Fix $a<0<b$ and $x \leq b$.
(i) For any $q > 0$, we have that 
\begin{align*}
\check{f}_S(x,a,b) :=\E_x\left(\int_{[0, \infty)} e^{-qt} \diff \tilde{L}_{r,S}^{a,b}(t)\right)= \frac {J_a^{(q,r)}(x)}  {(J_a^{(q,r)})'(b)}.
\end{align*}
(ii) If $q = 0$, it becomes infinity.
\end{proposition}
\begin{proof}
(i) By Remark \ref{remark_on_Y_r_tilde} (i),  \eqref{upcrossing_time_reflected}, and the strong Markov property, 
$\check{f}_S(a,a,b) =\E_a( e^{-q\eta_{0,a}^+} ) \check{f}_S(0,a,b) = \check{f}_S(0,a,b)  / {Z^{(q)}(-a)}$.
By this, Remark \ref{remark_on_Y_r_tilde} (ii), and the strong Markov property, together with Propositions \ref{prop_f_tilde_S} and \ref{prop_g_reflection_above}, for $x \leq b$,
%
\begin{align}\label{sin_triple_ref}
\begin{split}
\check{f}_S(x,a,b) 
&=\tilde{f}_S(x,a,b) +
\tilde{h}(x,a,b,0)
\check{f}_S(a,a,b) \\
&=\frac{I_a^{(q,r)}(x)}{(I_a^{(q,r)})'(b+)} +\Big( J_a^{(q,r)}(x)
-(J_a^{(q,r)})'(b)
\frac{I_a^{(q,r)}(x)}{(I_a^{(q,r)})'(b+)}\Big) \frac{\check{f}_S(0,a,b)}{Z^{(q)}(-a)}.
\end{split}
\end{align}
Now taking $x=0$ and solving for $\check{f}_S(0,a,b)$ (using \eqref{I_J_zero}),  we get,
$\check{f}_S(0,a,b) =Z^{(q)}(-a) / (J_a^{(q,r)})'(b)$.
Substituting this in \eqref{sin_triple_ref}, we have the claim.

(ii) It is immediate by \eqref{J_tilde_q_zero} upon taking $q \downarrow 0$ in (i).
\end{proof}


\begin{proposition}[Capital injections] \label{double_reflected_ci}Fix $a < 0 <b$ and $x\leq b$. 
(i) For any $q > 0$, we have
\begin{align*}
\check{j}(x,a,b) &:=\E_x\left(\int_{[0, \infty)} e^{-qt} \diff \tilde{R}_{r}^{a,b}(t)\right) =
-H_a^{(q,r)}(x)
+\frac{J_a^{(q,r)}(x)}{(J_a^{(q,r)})'(b)} (H_a^{(q,r)})'(b).
\end{align*}
(ii) When $q = 0$, it becomes infinity.
\end{proposition}

\begin{proof}
(i) First, by Remark \ref{remark_on_Y_r_tilde} (i), by modifying \eqref{j_hat_recursion},
$\check{j}(a,a,b)
=[ l^{(q)}(-a)+\check{j}(0,a,b)] / Z^{(q)}(-a)$.
In view of this, by Remark \ref{remark_on_Y_r_tilde} (ii), Corollary \ref{cor_overshoot_reflected}, and  the strong Markov property, we obtain a modification of \eqref{ci_2}:  for $x \leq b$,
\begin{align}\label{ci_2_reflected}
\begin{split}
&\check{j}(x,a,b)= 
\tilde{j}(x,a,b) +
\tilde{h}(x,a,b,0) \check{j}(a,a,b) \\
&=\frac {I_a^{(q,r)}(x)}   {(I_a^{(q,r)})'(b+)} \Big( (H^{(q,r)}_a)'(b) -
\frac {(J_a^{(q,r)})'(b)}  {Z^{(q)}(-a)} \check{j}(0,a,b) \Big) - H^{(q,r)}_a(x) +   \frac {J_a^{(q,r)}(x)} {Z^{(q)}(-a)} {\check{j}(0,a,b)}.
\end{split}
	\end{align}
Setting $x = 0$ and solving for $\check{j}(0,a,b)$ (using \eqref{I_J_zero} and \eqref{H_zero}), $\check{j}(0,a,b) = (H_a^{(q,r)})'(b) Z^{(q)}(-a)/(J_a^{(q,r)})'(b)$.
Substituting this back in \eqref{ci_2_reflected} we have the claim.
(ii) It is immediate by \eqref{J_tilde_q_zero} upon taking $q \downarrow 0$ in (i).

\end{proof}
\begin{remark}Recall \eqref{convergence_r_zero}.  As $r \downarrow 0$, we have the following.
	\begin{enumerate}
		\item By Proposition \ref{double_reflected_periodic}, $\check{f}_P(x,a,b)$ vanishes in the limit.
		\item By Proposition \ref{double_reflected_singular}, $\check{f}_S(x,a,b)$ converges to identity (4.3) in Theorem 1 of \cite{APP2007}.
		\item By Proposition \ref{double_reflected_ci}, $\check{j}(x,a,b )$ converges to identity (4.4) in Theorem 1 of \cite{APP2007}.
	\end{enumerate}
\end{remark}

\section{Proofs of Theorems for the bounded variation case} \label{section_proof_bounded}

In this section, we shall show Theorems \ref{prop_dividends}, \ref{proposition_upcrossing_time}, and \ref{proposition_laplace} for the case $X$ is of bounded variation. We shall use the following remark and lemma throughout the proofs.
\begin{remark}For $0 \leq t < \mathbf{e}_{r} \wedge \tau_0^-$, we have $X_r(t) = X(t)$ and $L_r(t) = 0$.  
\end{remark}

\begin{lemma} \label{lemma_H_X} 
For $q \geq 0$, $a < 0 < b$, and $x \leq b$,
\begin{align*}
\E_x\left( e^{-q \mathbf{e}_r} ;\mathbf{e}_r < \tau_0^- \wedge \tau_b^+  \right)  +  \E_x \Big( e^{-(q+r) \tau_0^-} \frac {W^{(q)}(X(\tau_0^-)-a)} {W^{(q)}(-a)} ; \tau_0^- < \tau_b^+  \Big)
&= I_a^{(q,r)}(x)  - \frac{W^{(q+r)}(x)}{W^{(q+r)}(b)} I_a^{(q,r)}(b).
\end{align*}
\end{lemma}

\begin{proof} 
As obtained in  (4.30) of \cite{APY} and by \eqref{laplace_in_terms_of_z},  
	for all $x \leq b$,
	\begin{align}\label{lap_exp}
\E_{x}
\left(e^{-q\mathbf{e}_{r}};
\mathbf{e}_{r}<\tau_b^+\wedge\tau_0^-\right) = \frac{r}{r+q}\mathbb{E}_x \left(1-e^{-(q+r) (\tau_b^+\wedge\tau_0^-)}\right) =r \Big(\frac {W^{(q+r)}(x)} {W^{(q+r)}(b)} \overline{W}^{(q+r)} (b)  -  \overline{W}^{(q+r)}(x) \Big).
	\end{align}
By summing this and \eqref{scale_function_overshoot_simplifying}, the result follows.
\end{proof}

\subsection{Proof of Theorem \ref{prop_dividends}}
We shall first show the following.
\begin{lemma} \label{lemma_X_e_r} 
For $b > 0$ and $x \leq b$, we have
 	\begin{align*}
	\E_{x}
\left(e^{-q\mathbf{e}_{r}}X(\mathbf{e}_{r});
\mathbf{e}_{r}<\tau_b^+\wedge\tau_0^-\right) 
  &=r \Big(\frac{ \overline{\overline{W}}^{(q+r)}(b)}{W^{(q+r)}(b)}W^{(q+r)}(x)-\overline{\overline{W}}^{(q+r)}(x)\Big).
 \end{align*}
 \end{lemma}
 \begin{proof}
 First we note that integration by parts gives for any $x\geq 0$, $\int_0^xyW^{(q+r)}(x-y) \diff y=\overline{\overline{W}}^{(q+r)}(x)$.
 	This implies, using the resolvent given in Theorem 8.7 of \cite{K}, the following
 	\begin{align*}
 	\E_{x}
 	\left(e^{-q\mathbf{e}_{r}}X(\mathbf{e}_{r});
 	\mathbf{e}_{r}<\tau_b^+\wedge\tau_0^-\right)&=r\int_0^by\left(W^{(q+r)}(x)\frac{W^{(q+r)}(b-y)}{W^{(q+r)}(b)}-W^{(q+r)}(x-y)\right) \diff y\\
 	&=r\Big(\frac{ \overline{\overline{W}}^{(q+r)}(b)}{W^{(q+r)}(b)}W^{(q+r)}(x)-\overline{\overline{W}}^{(q+r)}(x)\Big).
 	\end{align*}
 \end{proof}
For $x  \leq0$, by an application of the strong Markov property and \eqref{laplace_in_terms_of_z},
	 	 \begin{align} \label{f_a_b_recursion}
	 f(x,a,b)&=\E_x\left(e^{-q\tau_0^+};\tau_0^+<\tau_a^-\right)f(0,a,b) =\frac {W^{(q)}(x-a)} {W^{(q)}(-a)}f(0,a,b).
	 \end{align}
Using this and the strong Markov property, for $x \leq b$, 
	 \begin{align} \label{recursion_f}
	 \begin{split}
	 f(x,a,b)
	 &=\E_x\left(e^{-q\textbf{e}_r}X(\textbf{e}_r);\textbf{e}_r<\tau_0^-\wedge\tau_b^+ \right) +\E_x\left(e^{- (q+r) \tau_0^-}  {W^{(q)}(X(\tau_0^-)-a)}  ;\tau_0^-< \tau_b^+\right)  \frac {f(0,a,b)} {W^{(q)}(-a)} \\
	 &+\E_x\left(e^{-q\textbf{e}_r};\textbf{e}_r<\tau_0^-\wedge\tau_b^+\right) f(0,a,b).
	 \end{split}
	 \end{align}
By applying Lemmas \ref{lemma_H_X} and \ref{lemma_X_e_r} and \eqref{scale_function_overshoot_simplifying} in \eqref{recursion_f}, we obtain for all $x \leq b$,
	 \begin{align} \label{f_recursion}
	 \begin{split}
	 f(x,a,b) 
	& =- r \overline{\overline{W}}^{(q+r)}(x)   + I_a^{(q,r)}(x) f(0,a,b) + 	\frac {W^{(q+r)}(x)} {W^{(q+r)}(b)} \Big[ r \overline{\overline{W}}^{(q+r)}(b)  - I_a^{(q,r)}(b) f(0,a,b)\Big].
	 \end{split}
	 \end{align}
	 Setting $x = 0$ and solving for $f(0,a,b)$ (using \eqref{I_J_zero} and the fact that $W^{(q+r)}(0) > 0$ for the case of bounded variation as in \eqref{eq:Wqp0}), we have $f(0,a,b) = r \overline{\overline{W}}^{(q+r)}(b) /I_a^{(q,r)}(b)$. Substituting this back in \eqref{f_recursion}, we have the claim.
	 

\subsection{Proof of Theorem \ref{proposition_upcrossing_time}}
For $x \leq0$, similarly to \eqref{f_a_b_recursion} we obtain $g(x,a,b)=
 g(0,a,b)W^{(q)}(x-a) / {W^{(q)}(-a)}$.
 Now for $x \leq b$,  again by the strong Markov property, Lemma \ref{lemma_H_X}, and \eqref{scale_function_overshoot_simplifying},
 \begin{align} \label{g_recursion}
 \begin{split}
g(x,a,b)&=\mathbb{E}_x\left(e^{-q\mathbf{e}_r};\mathbf{e}_r<\tau_b^+\wedge\tau_0^-\right)g(0,a,b)+\mathbb{E}_x\left(e^{-q\tau_0^-}g(X(\tau_0^-),a,b);\tau_0^-<\mathbf{e}_r\wedge\tau_b^+\right)\\
&+ \E_x  (e^{-q \tau_b^+}; \tau_b^+ < \tau_0^- \wedge \mathbf{e}_r )\\
&= g(0,a,b) \Big(I_a^{(q,r)}(x) -  \frac{W^{(q+r)}(x)}{W^{(q+r)}(b)} I_a^{(q,r)}(b)\Big)
+ \frac {W^{(q+r)} (x)} {W^{(q+r)} (b)}.
\end{split}
\end{align} 
Setting $x = 0$ and using \eqref{I_J_zero}, $g(0,a,b) =  (I_a^{(q,r)}(b))^{-1}$. Substituting this in \eqref{g_recursion}, we have the result.

\subsection{Proof of Theorem \ref{proposition_laplace}}

(i) 
For $x \leq 0$, by using \eqref{laplace_in_terms_of_z},
\begin{align*}
h(x,a,b, \theta ) &= \E_x (e^{-q \tau_0^+}; \tau_0^+ < \tau_a^- ) h(0,a,b, \theta ) + \E_x \Big( e^{-q \tau_a^- - \theta  [a-X(\tau_a^-)]}; \tau_0^+ > \tau_a^- \Big) \\
&= \frac {W^{(q)}(x-a)} {W^{(q)}(-a)}[ h(0,a,b, \theta ) -Z^{(q)} (-a,\theta )]  +    Z^{(q)} (x-a,\theta ) .
\end{align*}
Using this and the strong Markov property, for all $x \leq b$, 
\begin{align*}
h(x,a,b,\theta ) &=\E_x\left(e^{- (q+r) \tau_0^-}  {h(X(\tau_0^-), a,b, \theta )}  ;\tau_0^-< \tau_b^+\right)   +\E_x\left(e^{-q\textbf{e}_r};\mathbf{e}_r<\tau_0^-\wedge\tau_b^+\right) h(0,a,b,\theta ) 
\end{align*}
where
\begin{align*} 
&\E_x\left(e^{- (q+r) \tau_0^-}  {h(X(\tau_0^-), a,b, \theta )}  ;\tau_0^-< \tau_b^+\right) \\
&=\E_x\left(e^{- (q+r) \tau_0^-}  \Big( \frac {W^{(q)}(X(\tau_0^-)-a)} {W^{(q)}(-a)}[ h(0,a,b, \theta ) -Z^{(q)} (-a,\theta )]  +    Z^{(q)} (X(\tau_0^-)-a,\theta ) \Big)  ;\tau_0^-< \tau_b^+\right).
\end{align*}

Hence, by Lemma  \ref{lemma_simplifying_formula_measure_changed},  \eqref{lap_exp}, and \eqref{scale_function_overshoot_simplifying},
\begin{align} \label{h_recurs}
\begin{split}
h(x,a,b,\theta ) 
&= \frac {h(0,a,b,\theta )- Z^{(q)} (-a, \theta )} {W^{(q)}(-a)}  \Big[ W^{(q,r)}_a(x)  - \frac {W^{(q+r)}(x)} {W^{(q+r)}(b)} W^{(q,r)}_a(b) \Big]  \\
&+  Z_a^{(q,r)}(x,\theta )- \frac {W^{(q +r)}(x)} {W^{(q+r)}(b)} Z_a^{(q,r)}(b,\theta ) + r \Big(  \frac {W^{(q+r)}(x)} {W^{(q+r)}(b)} \overline{W}^{(q+r)} (b)  -  \overline{W}^{(q+r)}(x) \Big) h(0,a,b,\theta ) \\
&= h(0,a,b,\theta ) \Big[ I_a^{(q,r)}(x) - \frac {W^{(q +r)}(x)} {W^{(q+r)}(b)} I_a^{(q,r)}(b) \Big] 
 +\hat{J}_a^{(q,r)}(x,\theta ) - \frac {W^{(q +r)}(x)} {W^{(q+r)}(b)} \hat{J}_a^{(q,r)}(b,\theta ).
 \end{split}
\end{align}
Setting $x = 0$, and using \eqref{I_J_zero} and \eqref{hat_J_zero},
we have $h(0,a,b,\theta )  = -  \hat{J}_a^{(q,r)}(b,\theta ) / I_a^{(q,r)}(b)$. Substituting this in \eqref{h_recurs}, we obtain the first identity (in terms of $\hat{J}_a^{(q,r)}$) in \eqref{h_a_b_expression}.  The last equality in \eqref{h_a_b_expression} holds by \eqref{J_J_tilde_relation}.

\section{Proofs for Theorems for the unbounded variation case} \label{section_proof_unbounded}

In this section, we shall show Theorems \ref{prop_dividends}, \ref{proposition_upcrossing_time}, and \ref{proposition_laplace} for the case $X$ is of unbounded variation. The proof is via excursion theory. We in particular use the recent results obtained in \cite{PPR15b} and the simplifying formula given in \cite{APY}.  We refer the reader to \cite{PPR15b} for detailed introduction and definitions regarding excursions away from zero for the case of spectrally negative \lev processes.

Fix $b > 0$ and $q > 0$.
Let us consider the event
\begin{align*}
E_B:=\{\tau_0^->\mathbf{e}_{r}\}\cup\{\zeta>\tau_b^+ \}\cup\{\zeta>\tau_a^{-}\}, 
\end{align*}
where $\mathbf{e}_{r}$ is an independent exponential clock with rate $r$,  $\zeta$ is the length of the excursion from the point it leaves $0$ and returns back to $0$. 
Due to the fact that $X$ is spectrally negative, once an excursion gets below zero, it stays until it ends at $\zeta$. 
That is, $E_B$ is the event in which (1) the exponential clock $\mathbf{e}_{r}$ that starts once the excursion becomes positive  rings before it downcrosses zero, (2) the excursion exceeds the level $b>0$, or (3) it goes below $a < 0$.  
	 
Now let us denote by $T_{E_B}$ the first time an excursion in the event $E_B$  occurs, and also denote by
\begin{align*}
l_{T_{E_B}} :=\sup\{t< T_{E_B}: X(t)=0\}, 
\end{align*}
the left extrema of the first excursion on $E_B$.
On the event 
$\{l_{T_{E_B}} < \infty \}$ we have
\begin{align*}
T_{E_B}=l_{T_{E_B}}+T_{E_B}\circ\Theta_{l_{T_{E_B}}}, 
\end{align*}
where we denote by $\Theta_t$ the shift operator at time $t\ge 0$.

Let $(e_{t}; t\geq 0)$ be the point process of excursions away from $0$ and $V:=\inf\{t>0: e_{t}\in E_B\}.$  
By, for instance, Proposition 0.2 in  \cite{B},  $(e_{t}, t<V)$ is independent of $(V,e_{V})$. The former is a Poisson point process with characteristic measure $\mathbf{n}(\cdot \cap E_B^c)$ and ${V}$ follows an exponential distribution with parameter $\mathbf{n}(E_B).$ Moreover, we have that $l_{T_{E_B}}=\sum_{s<V}\zeta(e_{s})$, where $\zeta(e_{s})$ denotes the lifetime of the excursion $e_{s}$.
Therefore,  the exponential formula  for Poisson point processes (see for instance Section 0.5 in \cite{B} or Proposition 1.12 in Chapter XII in \cite{RY}) and the independence between $(e_{t}, t<V)$ and  $(V,e_{V})$ imply 
\begin{align} \label{laplace_T_E_B}
\begin{split}
\e\Big( e^{-q l_{T_{E_B}}}\Big) &=\e\Big( \exp\Big\{-q \sum_{s<V}\zeta(e_{s})\Big\}\Big)=\mathbf{n}(E_B)\int_0^\infty e^{-s[\mathbf{n}(E_B)+\mathbf{n}(1-e^{-q\zeta}; E_B^c)]}\ud s\\
&=\frac{\mathbf{n}(E_B)}{\mathbf{n}(E_B)+\mathbf{n}\left(\mathbf{e}_q<\zeta, E_B^c\right)}=\frac{\mathbf{n}(E_B)}{\mathbf{n}(E_1)+\mathbf{n}(E_2)+\mathbf{n}(E_3)},
\end{split}
\end{align}
where $\mathbf{e}_q$ is an exponential random variable with parameter $q$ that is independent of $\mathbf{e}_r$ and $X$, and
 \begin{align*}
 E_1 &:= \{\mathbf{e}_q<\zeta\}\cup\{\tau_b^+<\zeta\}, \\ E_2&:= \{ \mathbf{e}_q>\zeta,\tau_a^-<\zeta < \tau_b^+ \}, \\ E_3 &:= \{\mathbf{e}_q>\zeta, \tau_0^- >\mathbf{e}_{r},\tau_a^- \wedge \tau_b^+>\zeta \}.
 \end{align*}
 To see how the last equality of \eqref{laplace_T_E_B} holds, we have
 \begin{multline*}
 \mathbf{n}(E_B)+\mathbf{n}\left(\mathbf{e}_q<\zeta, E_B^c\right) =  \mathbf{n}(\mathbf{e}_q < \zeta) + \mathbf{n}(\mathbf{e}_q > \zeta, E_B) \\= \mathbf{n}(E_1) - \mathbf{n} (\mathbf{e}_q > \zeta, \tau_b^+ < \zeta) + \mathbf{n}(\mathbf{e}_q > \zeta, E_B)
 =\mathbf{n}(E_1)+\mathbf{n}(E_2)+\mathbf{n}(E_3).
 \end{multline*}
 Now by Lemma 5.1 (i) and (ii)  in \cite{APY},  we have  \begin{itemize}
 	\item[(i)] $\displaystyle\mathbf{n}(E_1)
 	= e^{\Phi(q)b} / W^{(q)}(b)$,
 	\item[(ii)]$\displaystyle\mathbf{n}(E_2)=-\frac{1}{W^{(q)}(b)}\left(e^{\Phi(q) b}-\frac{W^{(q)}(b-a)}{W^{(q)}(-a)}\right)$.
 	\end{itemize}
 	On the other hand, we have the following; the proof is deferred to Appendix \ref{proof_lemma_n_E_3}.
	\begin{lemma} \label{lemma_n_E_3} 
	For $b, q > 0$, we have
	 	\begin{align} \label{E_3_excursion}
 	\mathbf{n}\left(E_3\right)=-\left(\frac{1}{W^{(q)}(b)}\frac{W^{(q)}(b-a)}{W^{(q)}(-a)}-\frac{1}{W^{(q+r)}(b)}\frac{W^{(q,r)}_a(b)}{W^{(q)}(-a)}\right).
 	\end{align}
	\end{lemma}
	
 	Hence $\mathbf{n}(E_1)+\mathbf{n}(E_2)+\mathbf{n}\left(E_3\right)={W^{(q,r)}_a(b)}/{[W^{(q)}(-a) W^{(q+r)}(b)]}$.
	This together with \eqref{laplace_T_E_B} gives 
	\begin{align} \label{ratio_excursion} 
\frac {\e ( e^{-q l_{T_{E_B}}} )} {\mathbf{n}(E_B)} &= W^{(q+r)}(b) \frac {W^{(q)}(-a)} {W^{(q,r)}_a(b)}.
\end{align}
	
	We now show the following lemma using the connections between $\mathbf{n}$ and the excursion measure of the process reflected at its infimum $\underline{\mathbf{n}}$, as obtained  in \cite{PPR15b}. 
\begin{lemma}  \label{lemma_some_excursion_results} Fix $b, q > 0$. (i)  We have $\mathbf{n}\left(e^{-q\mathbf{e}_r};\mathbf{e}_r<\tau_b^+\wedge\tau_0^-\right)
   =r {\overline{W}^{(q+r)}(b)} / {W^{(q+r)}(b)}$.
  
    (ii)  
    We have  $\mathbf{n}\left(e^{-q\mathbf{e}_r}X(\mathbf{e}_r);\mathbf{e}_r<\tau_b^+\wedge\tau_0^-\right) = r {\overline{\overline{W}}^{(q+r)}(b)} / {W^{(q+r)}(b)}$.
\end{lemma}
\begin{proof}
 	By a small modification of Theorem 3 (ii) in \cite{PPR15b}, using Proposition 1 in \cite{CD}, 
	  and by \eqref{lap_exp},
 	\begin{align*}
 	\mathbf{n}\left(e^{-q\mathbf{e}_r};\mathbf{e}_r<\tau_b^+\wedge\tau_0^-\right)&= 	\frac{r}{r+q}\mathbf{n}\left(1-e^{-(q+r) (\tau_b^+\wedge\tau_0^-)}\right) = \frac{r}{r+q} \underline{\mathbf{n}}\left(1-e^{-(q+r) (\tau_b^+\wedge\tau_0^-)}\right)\\
	&=\frac{r}{r+q}\lim_{x\downarrow 0}\frac{1}{W(x)}\mathbb{E}_x \left(1-e^{-(q+r) (\tau_b^+\wedge\tau_0^-)}\right) 
    =r\frac{\overline{W}^{(q+r)}(b)}{W^{(q+r)}(b)},
    \end{align*} 
where we use in the last equality that, as in the proof of Lemma 5.1 of \cite{APY}, 
\begin{align}
0 \leq \frac {W^{(q+r)}(x) - W(x)} {W(x)}\xrightarrow{x \downarrow 0} 0. \label{W_q_0_lim}
\end{align}

    Similarly we have using Lemma \ref{lemma_X_e_r} and \eqref{W_q_0_lim}, 
    \begin{align*} 
    \begin{split}
     		\mathbf{n}\left(e^{-q\mathbf{e}_r}X(\mathbf{e}_r);\mathbf{e}_r<\tau_b^+\wedge\tau_0^-\right)
     	&=\underline{\mathbf{n}}\left(e^{-q\mathbf{e}_r}X(\mathbf{e}_r);\mathbf{e}_r<\tau_b^+\wedge\tau_0^-\right)\\
     	&=\lim_{x\downarrow 0}\frac{1}{W(x)}\mathbb{E}_x\left(e^{-q\mathbf{e}_r}X(\mathbf{e}_r);\mathbf{e}_r<\tau_b^+\wedge\tau_0^-\right) = r\frac{\overline{\overline{W}}^{(q+r)}(b)}{W^{(q+r)}(b)}.
	\end{split}
     	\end{align*}
\end{proof}
We are now ready to show the theorems. We shall show for the case $q > 0$; the case $q=0$ holds by monotone convergence. For the rest of this section, let
	$\tilde{T}_0^- :=l_{T_{E_B}}+\tau_0^- \circ \Theta_{l_{T_{E_B}}}$.

\subsection{Proof of Theorems \ref{prop_dividends}} 
 	 By the definition of $l_{T_{E_B}}$, on the event $\{\tilde{T}_0^-  < (l_{T_{E_B}}+\mathbf{e}_{r}) \wedge \tau_b^+\}$, the excursion goes below $a$ and hence there is no contribution to $L_r$.  Therefore, by the strong Markov property,
 	\begin{align}\label{eqwexc1_ga}
 	f(0,a,b)&= f_0(0,a,b) + g_0(0,a,b) f(0,a,b),
 	\end{align}
 	where
 	\begin{align*}
 	f_0(0,a, b) &:= \E\Big( e^{-q (l_{T_{E_B}}+\mathbf{e}_{r})}X( l_{T_{E_B}}+\mathbf{e}_{r}); l_{T_{E_B}}+\mathbf{e}_{r}<\tilde{T}_0^- \wedge \tau_b^+\Big), \\ 
 	g_0(0,a,b)&:=\E\Big( e^{-q (l_{T_{E_B}}+\mathbf{e}_{r})}; l_{T_{E_B}}+\mathbf{e}_{r}<\tilde{T}_0^- \wedge \tau_b^+\Big).
 	\end{align*}
By the Master's formula in excursion theory (see for instance excursions straddling a terminal time in Chapter XII in Revuz and Yor \cite{RY}), Lemma \ref{lemma_some_excursion_results}, and \eqref{ratio_excursion}, and because $\{\mathbf{e}_{r}<\tau_0^- \wedge \tau_b^+ \} \subset E_B$, 
 	\begin{align}
 	f_0 (0, a,b) &= \frac {\e ( e^{-q l_{T_{E_B}}} )} {\mathbf{n} (E_B)}  {\mathbf{n}\left(e^{-q \mathbf{e}_{r}}X(\mathbf{e}_{r}); \mathbf{e}_{r}<\tau_0^- \wedge \tau_b^+ , E_B\right)}  
	= r\frac {W^{(q)}(-a)} {W^{(q,r)}_a(b)} \overline{\overline{W}}^{(q+r)} (b),\nonumber \\
 g_0 (0, a,b) &= \frac {\e ( e^{-q l_{T_{E_B}}} )} {\mathbf{n} (E_B)}  {\mathbf{n} ( e^{-q \mathbf{e}_{r}}; \mathbf{e}_{r}<\tau_0^-\wedge \tau_b^+ , E_B)}  
 = r \frac {W^{(q)}(-a)} {W^{(q,r)}_a(b)} \overline{W}^{(q+r)} (b). \label{g_0}
 \end{align}
Substituting these in \eqref{eqwexc1_ga},
 we obtain $f(0,a,b) = {r\overline{\overline{W}}^{(q+r)}(b)} /{I^{(q,r)}_a(b)}$.
Substituting this in \eqref{f_recursion} (which also holds for the unbounded variation case), we complete the proof.

\subsection{Proof of Theorem \ref{proposition_upcrossing_time}}
Similarly to \eqref{eqwexc1_ga},
\begin{align}\label{hb1}
g(0,a,b)=\mathbb{E}\left(e^{-q\tau_b^+};\tau_b^+<\tilde{T}_0^- \wedge (l_{T_{E_B}}+\mathbf{e}_{r})\right)+
g_0(0,a,b)
g(0,a,b).
\end{align}
By the Master's formula,  Lemma 5.1 (iv) in \cite{APY},  and \eqref{ratio_excursion}, and because $\{ \tau_b^+< \tau_0^- \} \subset E_B$,
\begin{align*}
\mathbb{E}\left(e^{-q\tau_b^+};\tau_b^+<\tilde{T}_0^- \wedge (l_{T_{E_B}}+\mathbf{e}_{r})\right)= \frac {\e\big( e^{-q l_{T_{E_B}}}\big)} {\mathbf{n} (E_B)} \mathbf{n}\left(e^{-(q+r)\tau_b^+};\tau_b^+< \tau_0^-, E_B\right) =  \frac   {W^{(q)}(-a)} {W^{(q,r)}_a(b)}.
\end{align*}
Substituting this and \eqref{g_0} in \eqref{hb1},  we have $g(0,a,b)=(I^{(q,r)}_a(b))^{-1}$.
Substituting this in \eqref{g_recursion} (which also holds for the unbounded variation case), we complete the proof.
\subsection{Proof of Theorems \ref{proposition_laplace}}

We shall first show the following using Theorem 5.1 in \cite{APY}; the proof is given in Appendix \ref{proof_excursion_simplification}.
\begin{lemma} \label{lemma_excursion_laplace}For $q > 0$, $a < 0 < b$, and $0 \leq \theta < \Phi(q)$, 
\begin{align*}
&\mathbf{n} \left(e^{-(q+r) \tau_0^-}\Big(Z^{(q)} (X(\tau_0^-)-a,\theta ) - Z^{(q)} (-a,\theta ) \frac {W^{(q)} (X(\tau_0^-)-a)} {W^{(q)}(-a)} \Big) ; \tau_0^-< \tau_b^+ \right)  =-\frac{\hat{J}_a^{(q,r)}(b, \theta)}{W^{(q+r)}(b)}.
\end{align*}
\end{lemma}

Using Lemma \ref{lemma_excursion_laplace}, we shall now give the proof of the theorem for $0 \leq  \theta < \Phi(q)$; the case $\theta \geq \Phi(q)$ holds by analytic continuation.
Using the strong Markov property we have that 
\begin{align}\label{h1}
h(0,a,b,\theta )&=g_0(0,a, b)
h(0,a,b, \theta ) 
+\mathbb{E}\left(e^{-q [\tilde{T}_0^- + \tau_a^- \circ \Theta_{\tilde{T}_0^-}] - \theta [a- X(\tilde{T}_0^- + \tau_a^- \circ \Theta_{\tilde{T}_0^-})]}; \tilde{T}_0^-<(l_{T_{E_B}}+\mathbf{e}_{r}) \wedge\tau_b^+\right). \end{align}
By Master's formula and \eqref{ratio_excursion},
\begin{align*}
&\mathbb{E}\Big(e^{-q [\tilde{T}_0^- + \tau_a^- \circ \Theta_{\tilde{T}_0^-}] - \theta  [a-X(\tilde{T}_0^- + \tau_a^- \circ \Theta_{\tilde{T}_0^-})]}; \tilde{T}_0^-<(l_{T_{E_B}}+\mathbf{e}_{r})\wedge\tau_b^+\Big) \\ &= \frac {\E (e^{-q l_{T_{E_B}}} )} {\mathbf{n} (E_B)}   {\mathbf{n} \left(e^{-q \tau_a^- -  \theta  [a-X (\tau_a^-)]}; \tau_0^-< \mathbf{e}_{r} \wedge\tau_b^+, E_B\right)} \\ &= W^{(q+r)}(b) \frac {W^{(q)}(-a)} {W^{(q,r)}_a(b)} \mathbf{n} \left(e^{-q \tau_a^- -  \theta  [a-X (\tau_a^-)]}; \tau_0^-< \mathbf{e}_{r} \wedge\tau_b^+, \tau_a^- < \tau_0^- + \tau_0^+ \circ \Theta_{\tau_0^-} \right).
\end{align*}
Here, by the strong Markov property, \eqref{laplace_in_terms_of_z}, and Lemma \ref{lemma_excursion_laplace}, 
\begin{align*}
&\mathbf{n} \left(e^{-q \tau_a^- -  \theta  [a-X (\tau_a^-)]}; \tau_0^-< \mathbf{e}_{r} \wedge\tau_b^+, \tau_a^- < \tau_0^- + \tau_0^+ \circ \Theta_{\tau_0^-}  \right) \\ &= \mathbf{n} \left(e^{-q \tau_0^-}\mathbb{E}_{X(\tau_0^-)} \left(e^{-q\tau_a^-- \theta [a-X (\tau_a^-)]};\tau_a^-<\tau_0^+\right); \tau_0^-<\mathbf{e}_r\wedge\tau_b^+\right) \\
&= \mathbf{n} \left(e^{-(q+r) \tau_0^-}\mathbb{E}_{X(\tau_0^-)} \left(e^{-q\tau_a^-- \theta [a-X (\tau_a^-)]};\tau_a^-< \tau_0^+ \right); \tau_0^-< \tau_b^+\right) 
=-\frac{\hat{J}_a^{(q,r)}(b,\theta )}{W^{(q+r)}(b)}. 
\end{align*}

Substituting these and \eqref{g_0} in \eqref{h1}, we obtain that $h(0,a,b,\theta ) = - \hat{J}_a^{(q,r)}(b,\theta ) / I_a^{(q,r)}(b)$.
Using this expression in \eqref{h_recurs} (which also holds for the unbounded variation case), we complete the proof.

\appendix

\section{Proofs regarding simplifying formulae}

As in \cite{LRZ}, for any $\alpha \geq 0$,  let $\mathcal{V}_0^{(\alpha)}$ be the set of measurable functions $v_{\alpha}: \R \to \R$ 
\begin{align*}
\E_x \big(e^{- \alpha \tau_0^-} v_{\alpha} (X(\tau_0^-)); \tau_0^- < \tau_b^+ \big) = v_{\alpha} (x) - \frac {W^{(\alpha)}(x)} {W^{(\alpha)}(b)} v_{\alpha} (b), \quad x \leq b.
\end{align*}
We shall further define $\tilde{\mathcal{V}}_0^{(\alpha)}$ to be the set of positive measurable functions $v_{\alpha}(x)$ that satisfy conditions (i) or (ii) in Lemma 2.1 of \cite{LRZ}, which state as follows: 
\begin{enumerate}
\item[(i)] For the case $X$ is of bounded variation,  $v_{\alpha} \in \mathcal{V}_0^{(\alpha)}$ and there exists large enough $\lambda$ such that $\int_0^\infty e^{- \lambda z} v_{\alpha} (z) \diff z < \infty$.
\item[(ii)] For the case $X$ is of unbounded variation, there exist a sequence of functions $v_{\alpha,n}$ that converge to $v_{\alpha}$ uniformly on compact sets, where $v_{\alpha,n}$ belongs to the class $\tilde{\mathcal{V}}_0^{(\alpha)}$ for the process $X^n$; here $(X^n; n \geq 1)$ is a sequence of spectrally negative \lev processes of bounded variation that converge to $X$ almost surely uniformly on compact time intervals (which can be chosen as in, for example, page 210 of \cite{B}).
\end{enumerate}

Fix any $a < 0$.  By Lemma 2.2 of \cite{LRZ} and spatial homogeneity,
\begin{align}
y \mapsto W^{(\alpha)}(y-a) \in \tilde{\mathcal{V}}_0^{(\alpha)} \quad \textrm{and} \quad y \mapsto Z^{(\alpha)}(y-a) \in \tilde{\mathcal{V}}_0^{({\alpha})}. \label{W_Z_in_V}
\end{align}
Lemma 2.1 of \cite{LRZ} shows that, for all $\alpha, \beta \geq 0$, $v_{\alpha} \in \tilde{\mathcal{V}}_0^{(\alpha)}$ and $x \leq b$,
\begin{multline} \label{LRZ_identity}
\E_x \big(e^{-\beta \tau_0^-} v_{\alpha} (X(\tau_0^-)); \tau_0^- < \tau_b^+ \big) = v_{\alpha} (x) - (\alpha- \beta) \int_0^x W^{(\beta)} (x-y) v_{\alpha}(y) \diff y \\ - \frac {W^{(\beta)}(x)} {W^{(\beta)}(b)} \Big( v_{\alpha}(b) - (\alpha-\beta) \int_0^b W^{(\beta)}(b-y) v_{\alpha} (y) \diff y\Big).
\end{multline}
Similar results under the excursion measure have been obtained in Theorem 5.1 of \cite{APY} (see \eqref{simplifying_formula_excursion} below). 

In the following proofs, we need measure-changed versions of these theorems.
For $\theta  \geq 0$, let $\p^{\theta }$ be the measure under the Esscher transform 
\begin{align}
\left. \frac {\diff \p^{\theta }} {\diff \p}\right|_{\mathcal{F}_t} = \exp(\theta  X(t) - \psi(\theta ) t), \quad t \geq 0, \label{change_of_measure}
\end{align}
and $W_{\theta }$ and $Z_{\theta }$ be the corresponding scale functions.  It is well known that 
\begin{align}
W_{\theta }^{(q-\psi(\theta ))}(y) = e^{-{\theta }y} W^{(q)}(y), \quad y \in \R, \; q \geq 0. \label{W_change_measure}
\end{align}
Hence, we have
\begin{align}
Z_{\theta }^{(q-\psi(\theta ))} (y) = e^{-\theta y} Z^{(q)} (y, \theta), \quad y \in \R, \; \theta \geq 0. \label{Z_change_measure}
\end{align}


\subsection{Proof of Lemma \ref{lemma_simplifying_formula_measure_changed}} \label{proof_simplifying_formula}

\underline{Proof of \eqref{Z_overshoot_identity}:} Fix $a < 0$. Using the measure-changed version of \eqref{LRZ_identity}: if $v_{q-\psi(\theta )} \in \tilde{\mathcal{V}}_0^{(q-\psi(\theta ))}$ under $\p^\theta$, then for $x \leq b$,
\begin{align*}
e^{-\theta  (x-a)}\E_x & \big(e^{-(q+r) \tau_0^- + \theta  (X(\tau_0^-)-a)} v_{q-\psi(\theta )} (X(\tau_0^-)); \tau_0^- < \tau_b^+ \big) = \E_x^{\theta }  \big(e^{-(q+r - \psi(\theta )) \tau_0^-} v_{q-\psi(\theta )} (X(\tau_0^-)); \tau_0^- < \tau_b^+ \big) \notag\\&= v_{q-\psi(\theta )} (x) +r \int_0^x W_{\theta }^{(q+r-\psi(\theta ))} (x-y) v_{q-\psi(\theta )}(y) \diff y \\ &- \frac {W_{\theta }^{(q+r-\psi(\theta ))}(x)} {W_{\theta }^{(q+r-\psi(\theta ))}(b)} \Big( v_{q-\psi(\theta )}(b) +r \int_0^b  W_{\theta }^{(q+r-\psi(\theta ))} (b-y) v_{q-\psi(\theta )} (y) \diff y\Big).
\end{align*}
Hence, because  $y \mapsto Z_{\theta }^{(q-\psi(\theta ))} (y -a) \in \tilde{\mathcal{V}}_0^{(q-\psi(\theta ))}$ under $\p^{\theta }$ as in \eqref{W_Z_in_V}, with $\mathcal{M}_{a, \theta}^{(q,r)}$ the measure-changed version of \eqref{operator_M} such that
\begin{align*}
\mathcal{M}_{a, \theta}^{(q,r)}f (x) := f (x-a) +r \int_0^x W_{\theta }^{(q+r-\psi(\theta ))} (x-y) f(y-a) \diff y, \quad x \in \R,
\end{align*} 
the left hand side of \eqref{Z_overshoot_identity} equals
\begin{align*}
&\E_x \big(e^{-(q+r) \tau_0^- + \theta  (X(\tau_0^-)-a)} Z_{\theta }^{(q-\psi(\theta ))} (X(\tau_0^-)-a); \tau_0^- < \tau_b^+ \big) \\ &= e^{\theta (x-a)} \Big[\mathcal{M}_{a, \theta}^{(q,r)} Z_{\theta }^{(q-\psi(\theta ))}(x) - \frac {W_{\theta }^{(q+r-\psi(\theta ))}(x)} {W_{\theta }^{(q+r-\psi(\theta ))}(b)}\mathcal{M}_{a, \theta}^{(q,r)} Z_{\theta }^{(q-\psi(\theta ))}(b)\Big] = Z^{(q,r)}_a(x,\theta ) - \frac {W^{(q+r)}(x)} {W^{(q+r)}(b)}  Z^{(q,r)}_a(b,\theta ),
\end{align*}
where the last equality holds because, for all $y \in \R$, by \eqref{Z_change_measure}, 
\begin{multline} \label{Z_v_operator_simplification}
\mathcal{M}_{a, \theta}^{(q,r)} Z_{\theta }^{(q-\psi(\theta ))} (y) = e^{-\theta  (y-a)}Z^{(q)} (y-a, \theta ) +r \int_0^y e^{-\theta  (y-z)}W^{(q+r)} (y-z) e^{-\theta  (z-a)}Z^{(q)} (z-a, \theta ) \diff z \\
= e^{-\theta  (y-a)} \Big[ Z^{(q)} (y-a, \theta ) +r \int_0^y W^{(q+r)} (y-z) Z^{(q)} (z-a, \theta) \diff z \Big] = e^{-\theta  (y-a)} Z^{(q,r)}_a(y,\theta ).
\end{multline}

\underline{Proof of \eqref{Z_overshoot_identity_reflected}:} We first generalize the results for  Theorem 6.1 of \cite{APY}.  The result (ii) is then immediate by setting $\beta = q+r$ and $\alpha = q - \psi(\theta)$ and  observing that
$y \mapsto Z_\theta^{(q-\psi(\theta ))}(y -a) \in \tilde{\mathcal{V}}_0^{(q-\psi(\theta ))}$ under $\p^{\theta}$,
and by \eqref{Z_change_measure},
\begin{align*}
\E_x \big(e^{-(q+r) \tilde{\tau}_{0,b}^-} Z^{(q)} (\overline{Y}^b(\tilde{\tau}_{0,b}^-)-a, \theta )  \big) = e^{-\theta a}\E_x \big(e^{-(q+r) \tilde{\tau}_{0,b}^- + \theta  \overline{Y}^b(\tilde{\tau}_{0,b}^-)} Z_{\theta }^{(q-\psi(\theta ))} (\overline{Y}^b(\tilde{\tau}_{0,b}^-)-a)  \big). \end{align*}

\begin{theorem} \label{lemma_simplifying_reflected} Fix $\alpha, \beta \geq 0$, $\theta  \geq 0$, and $b > 0$. Suppose $v_{\alpha}: \R \to [0, \infty)$  and belongs to $\tilde{\mathcal{V}}_0^{(\alpha)}$ under $\p^{\theta }$.  Assume also that $v_{\alpha}$ is right-hand differentiable at $b$ and $\sup_{0\le y\leq b}\int_{(-\infty, -1]} v_{\alpha} (y + u) e^{\theta  u} \Pi (\diff u) < \infty$.  In addition, for the case of unbounded variation, in (ii) for the definition of $\tilde{\mathcal{V}}_0^{(\alpha)}$ above, $v_{\alpha,n}'(b +) \xrightarrow{n \uparrow \infty}v_{\alpha}'(b +)$.
	Then, for  $q \geq 0$ and $x \leq b$,  
	\begin{align*}
	\begin{split}
	&\E_x \Big( e^{-\beta \tilde{\tau}_{0,b}^-  + \theta  \overline{Y}^b (\tilde{\tau}_{0,b}^- )} v_{\alpha}(\overline{Y}^b(\tilde{\tau}_{0,b}^- )) \Big)  \\
&=\frac {W^{(\beta)}(x)} {W^{(\beta)\prime}(b+)} \Big[  - e^{\theta b} (v_{\alpha}'(b +) - \theta  v_{\alpha}(b) )+ (\alpha-\beta + \psi(\theta)) \Big( \int_0^b e^{\theta y} W^{(\beta)\prime} (b-y) v_{\alpha}(y) \diff y + e^{\theta b} W^{(\beta)} (0) v_{\alpha}(b) \Big) \Big] \\
	&+ e^{\theta x}v_{\alpha}(x) - (\alpha-\beta + \psi(\theta)) \int_0^x  W^{(\beta)} (x-y) e^{\theta y} v_{\alpha}(y) \diff y.
	\end{split}
	\end{align*}	
	\end{theorem}
\begin{proof} 
We consider the case of bounded variation.  It can be extended to the unbounded variation case by approximation as in the proof of  Theorem 6.1 of \cite{APY}. 
We also focus on the case $0 \leq x \leq b$; the case $x < 0$ is immediate.

	Using the resolvent given in Theorem 1 (ii) of \cite{P2004}, and the compensation formula, we have 
	\begin{align} \label{resolvent_expression_reflected}
	\begin{split}
	&\E_x \big( e^{-\beta \tilde{\tau}_{0,b}^-  + \theta  \overline{Y}^b(\tilde{\tau}_{0,b}^- )} v_{\alpha}(\overline{Y}^b(\tilde{\tau}_{0,b}^- )) \big) \\&= \int_0^\infty \int_{(-\infty, -y)} e^{\theta  (y+u)} v_{\alpha}(y + u) \Pi (\diff u)  \Big[ \frac {W^{(\beta)\prime}(b-y)} {W^{(\beta)\prime}(b +)} W^{(\beta)}(x) - W^{(\beta)}(x-y)\Big] \diff y \\
	&+  W^{(\beta)}(x) \frac {W^{(\beta)}(0)} {W^{(\beta)\prime}(b+)} \int_{(-\infty, -b )} e^{\theta  (b+u)} v_{\alpha}(b + u ) \Pi (\diff u).
	\end{split}
	\end{align}
	By (19) of \cite{LoRZ}, \eqref{W_change_measure}, and because $v_{\alpha}$ belongs to $\tilde{\mathcal{V}}_0^{(\alpha)}$ under $\p^\theta$ by assumption, we have
	\begin{align*}
	&e^{-\theta b}\int_0^{b} W^{(\beta)}(b-y) \int_{(-\infty, -y)} v_{\alpha}(y + u) e^{\theta  (y+u)}\Pi (\diff u)  \diff y \\
&\qquad =\int_0^{b} W^{(\beta-\psi(\theta ))}_\theta(b-y) \int_{(-\infty, -y)} v_{\alpha}(y + u) e^{\theta  u} \Pi (\diff u)  \diff y \\ &\qquad = c v_{\alpha} (0) W^{(\beta-\psi(\theta ))}_{\theta }(b) - v_{\alpha}(b) + (\alpha-\beta + \psi(\theta)) \int_0^b W^{(\beta-\psi(\theta ))}_{\theta } (b-y) v_{\alpha}(y) \diff y, 
	\end{align*}
where we recall $c=W_\theta^{(\alpha)}(0)$ as in \eqref{eq:Wqp0}.
Therefore, by \eqref{W_change_measure},
		\begin{multline*}
\int_0^{b} W^{(\beta)}(b-y) \int_{(-\infty, -y)} v_{\alpha}(y + u) e^{\theta  (y+u)}\Pi (\diff u)  \diff y \\ = c v_{\alpha} (0) W^{(\beta)}(b) - e^{\theta b}v_{\alpha}(b) + (\alpha-\beta + \psi(\theta)) \int_0^b e^{\theta y} W^{(\beta)} (b-y) v_{\alpha}(y) \diff y. 
	\end{multline*}
Taking the right-hand derivative with respect to $b$,
		\begin{align*} 
		\begin{split}
&\int_0^{b} W^{(\beta) \prime} (b-y) \int_{(-\infty, -y)} v_{\alpha}(y + u) e^{\theta  (y+u)}\Pi (\diff u)  \diff y + W^{(\beta) }(0) \int_{(-\infty, -b)} v_{\alpha}(b + u) e^{\theta  (b+u)}\Pi (\diff u)   \\
&=\frac {\partial_+} {\partial_+ b}\int_0^{b} W^{(\beta)}(b-y) \int_{(-\infty, -y)} v_{\alpha}(y + u) e^{\theta  (y+u)}\Pi (\diff u)  \diff y
 \\&= c v_{\alpha} (0) W^{(\beta)\prime}(b +) - e^{\theta  b}(v_{\alpha}'(b) + \theta  v_{\alpha}(b) ) + (\alpha-\beta  + \psi(\theta)) \Big[ \int_0^b e^{\theta  y} W^{(\beta)\prime} (b-y) v_{\alpha}(y) \diff y + e^{\theta  b} W^{(\beta)} (0) v_{\alpha}(b) \Big];
 \end{split}
	\end{align*}
to see how the derivative can be interchanged over the integral in the first equality, see the proof of Theorem 6.1 of \cite{APY}.
	Hence, substituting this in \eqref{resolvent_expression_reflected}, and after simplification, we have the claim.
\end{proof}

\subsection{Proof of Lemma \ref{lemma_n_E_3}} \label{proof_lemma_n_E_3}
Using the fact $\{ \mathbf{e}_q > \tau_0^- > \mathbf{e}_r \} \sqcup \{ \mathbf{e}_q \wedge \mathbf{e}_r > \tau_0^-  \} = \{ \mathbf{e}_q >  \tau_0^- \}$ and that $\mathbf{e}_q \wedge \mathbf{e}_r$ is exponentially distributed with parameter $q+r$,
 	\begin{align*}
 	&\mathbf{n}\left(E_3\right)
 	=\mathbf{n}\left(\left(e^{-q\tau_0^-}-e^{-(q+r)\tau_0^-}\right)\E_{X(\tau_0^-)} 	\left(e^{-q\tau_0^+};\tau_0^+<\tau_a^-\right); \tau_0^-<\tau_b^+\right).
 	\end{align*}
Here by a small modification of Theorem 3 (ii) in \cite{PPR15b} and using Proposition 1 in \cite{CD} and \eqref{LRZ_identity}, the right hand side equals
\begin{align*}
&\underline{\mathbf{n}}\left(\left(e^{-q\tau_0^-}-e^{-(q+r)\tau_0^-}\right)\E_{X(\tau_0^-)} 	\left(e^{-q\tau_0^+};\tau_0^+<\tau_a^-\right); \tau_0^-<\tau_b^+\right)\\
&=\lim_{x \downarrow 0}W(x)^{-1}\E_x\left(\left(e^{-q\tau_0^-}-e^{-(q+r)\tau_0^-}\right)\E_{X(\tau_0^-)} 	\left(e^{-q\tau_0^+};\tau_0^+<\tau_a^-\right); \tau_0^-<\tau_b^+\right)
\\ &= \lim_{x \downarrow 0}  \frac 1 {W(x) W^{(q)}(-a)} \\ &\qquad \times \Big[ - r \int_0^x W^{(q+r)} (x-y)  {W^{(q)} (y-a)} \diff y -\Big( \frac {W^{(q)}(x) W^{(q)}(b-a)} {W^{(q)}(b)} - \frac { W^{(q+r)}(x) W^{(q,r)}_a(b)} {W^{(q+r)}(b)} \Big) \Big], 
\end{align*}
which equals the right hand side of \eqref{E_3_excursion} by \eqref{W_q_0_lim}, as desired.

\subsection{Proof of Lemma \ref{lemma_excursion_laplace}} \label{proof_excursion_simplification}Let $\mathbf{n} ^{\theta }$ be the excursion measure under the Esscher transform  \eqref{change_of_measure}. By Theorem 5.1 and Remark 5.1 in \cite{APY}, if $v_{q-\psi(\theta )} \in \tilde{\mathcal{V}}_0^{(q-\psi(\theta ))}$, $v_{q-\psi(\theta )}(0) = 0$ and it is differentiable at $0$, then
\begin{multline} \label{simplifying_formula_excursion}
\mathbf{n}^{\theta}\left(e^{-(q+r-\psi(\theta ))\tau_0^-}v_{q-\psi(\theta )}(X(\tau_0^-));\tau_0^-<\tau_b^+\right)
=- \frac {v_{q-\psi(\theta )}(b)+r \int_0^bW_{\theta }^{(q+r-\psi(\theta ))}(b-y)v_{q-\psi(\theta )}(y) \diff y } {W_{\theta }^{(q+r-\psi(\theta ))} (b)}.
\end{multline}
 By \eqref{Z_change_measure} and  because $y \mapsto Z_{\theta }^{(q-\psi(\theta ))} (y-a) - Z^{(q)} (-a,\theta ) \ {W^{(q-\psi(\theta ))}_\theta (y-a)} / {W^{(q)}(-a)} \in \tilde{\mathcal{V}}_0^{(q-\psi(\theta ))}$ under $\p^\theta$ (by \eqref{W_Z_in_V} and because $\tilde{\mathcal{V}}_0^{(q-\psi(\theta ))}$ is a linear space),
\begin{align*}
&\mathbf{n} \left(e^{-(q+r) \tau_0^-}\Big(Z^{(q)} (X(\tau_0^-)-a,\theta ) - Z^{(q)} (-a,\theta ) \frac {W^{(q)} (X(\tau_0^-)-a)} {W^{(q)}(-a)} \Big) ; \tau_0^-< \tau_b^+ \right)  \\
&= e^{\theta  a} \mathbf{n}^{\theta } \left(e^{-(q+r - \psi(\theta )) \tau_0^-}\Big(Z_{\theta }^{(q-\psi(\theta ))} (X(\tau_0^-)-a) - Z^{(q)} (-a,\theta ) \frac {W^{(q-\psi(\theta ))}_{\theta } (X(\tau_0^-)-a)} {W^{(q)}(-a)} \Big); \tau_0^-< \tau_b^+\right)  \\
&= - \frac {e^{\theta  a}} {W_{\theta }^{(q+r-\psi(\theta ))}(b)} \left( \mathcal{M}_{a, \theta}^{(q,r)} Z_{\theta }^{(q-\psi(\theta ))}(b)-\frac{Z^{(q)}(-a,\theta )}{W^{(q)}(-a)} \mathcal{M}_{a, \theta}^{(q,r)}  W^{(q-\psi(\theta ))}_\theta (b)\right),
\end{align*}
which simplifies to $- \hat{J}_a^{(q,r)}(b,\theta)/ {W^{(q+r)}(b)}$ by \eqref{Z_v_operator_simplification} and because
\begin{multline*}
\mathcal{M}_{a, \theta}^{(q,r)} W_{\theta }^{(q-\psi(\theta ))} (b) = e^{-\theta  (b-a)}W^{(q)} (b-a) +r \int_0^{b} e^{-\theta  (b-z)}W^{(q+r)} (b-z) e^{-\theta  (z-a)}W^{(q)} (z-a) \diff z \\
= e^{-\theta  (b-a)} \Big[ W^{(q)} (b-a) +r \int_0^{b} W^{(q+r)} (b-z) W^{(q)} (z-a) \diff z \Big] = e^{-\theta  (b-a)} W^{(q,r)}_a(b).
\end{multline*}

\section{Proofs of Corollaries}
Before we provide the proofs of the corollaries we first state the following convergence results that will be used  throughout this appendix. By \eqref{W^{(q)}_limit}, it is immediate that, for $q \geq 0$,
\begin{align} \label{conv_ratio_scale_functions}
\lim_{b \uparrow \infty}\frac {W^{(q+r)\prime}(b+)} {W^{(q+r)}(b)} = \lim_{b \uparrow \infty}\frac {W^{(q+r)}(b)} {\overline{W}^{(q+r)}(b)} =\Phi(q+r) \quad \textrm{and} \quad  \lim_{b \uparrow \infty}\frac {W^{(q+r)}(b)} {\overline{\overline{W}}^{(q+r)}(b)} =\Phi^2(q+r). 
\end{align}
Also note that we can write, by (5) of \cite{LRZ} and (3.4) of \cite{PY15a}, 
\begin{align} \label{M_expression_alternative}
	\begin{split}
		W^{(q,r)}_a(x) 
		&= W^{(q+r)}(x-a)-r \int_0^{-a}W^{(q+r)}(x-u-a)W^{(q)}(u) \diff u, \\
		\overline{Z}^{(q,r)}_a(x) &=\overline{Z}^{(q+r)}(x-a)-r\int_0^{-a}W^{(q+r)}(x-u-a)\overline{Z}^{(q)}(u)\diff u.
	\end{split}
\end{align}

\begin{lemma} \label{lemma_convergence_W_a} Fix $q \geq 0$.
(i) For $x \in \R$, we have
$\lim_{a \downarrow -\infty}[{W^{(q,r)}_a(x)} / {W^{(q)}(-a)}] = Z^{(q+r)} (x, \Phi(q))$. \\
(ii) For $a < 0$, we have $\lim_{b \uparrow \infty}[ {W^{(q,r)}_a(b)} / {W^{(q+r)}(b)}] = Z^{(q)} (-a, \Phi(q+r))$. \\
(iii) For $a < 0$ and $0 \leq \theta  < \Phi(q)$, we have $\lim_{b\uparrow\infty}[Z^{(q,r)}_a(b,\theta )/W^{(q+r)}(b)]=\tilde{Z}^{(q,r)}(-a,\theta )$. \\
(iv) For  $a < 0$, we have
\begin{align*}
	\lim_{b\uparrow \infty}\frac{\overline{W}_a^{(q,r)}(b)}{W^{(q+r)}(b)}=\frac{\tilde{Z}^{(q,r)}(-a)}{q}-\frac{r}{q\Phi(q+r)},
	\end{align*}
where it is understood for the case $q = 0$ that it goes to infinity.
\\
(v) For $a < 0$, we have 
\begin{align*}
		\lim_{b \uparrow\infty}\frac{\overline{Z}_a^{(q,r)}(b)}{W^{(q+r)}(b)}&=\frac{r}{\Phi(q+r)}\overline{Z}^{(q)}(-a)+ \frac {\tilde{Z}^{(q,r)}(-a)} {\Phi(q+r)}. 
		\end{align*}
\end{lemma}
\begin{proof}  (i) By \eqref{W^{(q)}_limit}, we have
	\begin{align*}
		\lim_{a \downarrow -\infty}\frac {W^{(q,r)}_a(x)} {W^{(q)}(-a)} &=\lim_{a \downarrow -\infty} \frac {W^{(q)} (x-a) +r \int_0^x W^{(q+r)} (x-y) W^{(q)}(y-a) \diff y} {W^{(q)}(-a)} \\ &= e^{\Phi(q) x} + r \int_0^x e^{\Phi(q) y }W^{(q+r)} (x-y)  \diff y
		= Z^{(q+r)} (x, \Phi(q)).
	\end{align*}
	(ii) By \eqref{W^{(q)}_limit} and \eqref{M_expression_alternative}, we have
	\begin{align*}
		\lim_{b \uparrow \infty}\frac {W^{(q,r)}_a(b)} {W^{(q+r)}(b)} = e^{-\Phi(q+r) a} \Big(1 - r \int_0^{-a} e^{-\Phi(q+r) y} W^{(q)} (y) \diff y \Big) = Z^{(q)} (-a, \Phi(q+r)).
	\end{align*}
(iii) We have\begin{align*}
\lim_{b\uparrow\infty}\frac{Z^{(q,r)}_a(b,\theta )}{W^{(q+r)}(b)}=\lim_{b\uparrow\infty}\frac{Z^{(q)}(b-a,\theta )+r\int_0^bW^{(q+r)}(b-y)Z^{(q)}(y-a,\theta )\diff y}{W^{(q+r)}(b)}.
\end{align*}
Because $\theta <\Phi(q) < \Phi(q+r)$ and by \eqref{W^{(q)}_limit}, we have $\lim_{b\uparrow\infty}Z^{(q)}(b-a,\theta ) / W^{(q+r)}(b) = 0$.
On the other hand, by \eqref{W^{(q)}_limit}, 
\begin{multline*}
\lim_{b\uparrow\infty}\frac{\int_0^bW^{(q+r)}(b-y)Z^{(q)}(y-a,\theta )\diff y}{W^{(q+r)}(b)}=\int_0^{\infty}e^{-\Phi(q+r)y}Z^{(q)}(y-a,\theta )\diff y \\
=\int_0^{\infty}e^{-\Phi(q+r)y + \theta  (y-a)}\diff y+(q-\psi(\theta ))\int_0^{\infty}e^{-\Phi(q+r)y + \theta  (y-a)} \int_0^{y-a}e^{-\theta  z}W^{(q)}(z)\diff z \diff y.
\end{multline*}
For the first term we have $\int_0^{\infty}e^{-\Phi(q+r)y}e^{\theta  (y-a)}\diff y= e^{-\theta a} / {(\Phi(q+r)-\theta )}$.
For the second term, using Fubini's theorem, 
\begin{align*}
&\int_0^{\infty}e^{-\Phi(q+r)y+\theta (y-a)} \int_0^{y-a}e^{-\theta z}W^{(q)}(z) \diff z \diff y = e^{-a \theta}\int_0^{\infty}e^{-[\Phi(q+r) - \theta]y} \int_0^{y-a}e^{-\theta z}W^{(q)}(z) \diff z \diff y \\
&=  e^{-a \theta} \Big( \int_0^{-a} \int_0^{\infty} e^{-[\Phi(q+r) - \theta]y} e^{-\theta z}W^{(q)}(z)  \diff y \diff z +  \int_{-a}^{\infty} \int_{z+a}^{\infty} e^{-[\Phi(q+r) - \theta]y} e^{-\theta z}W^{(q)}(z)  \diff y \diff z \Big)\\
&=  e^{-a \theta} \Big( \int_0^{-a} \frac 1 {\Phi(q+r)-\theta} e^{-\theta z}W^{(q)}(z)   \diff z +  \int_{-a}^{\infty} \frac {e^{-[\Phi(q+r) - \theta] (z+a)}} {\Phi(q+r)-\theta} e^{-\theta z}W^{(q)}(z)   \diff z \Big) \\
&=   \frac {e^{-a \theta}} {\Phi(q+r)-\theta}  \int_0^{-a}  e^{-\theta z}W^{(q)}(z)   \diff z +  \frac {e^{-\Phi(q+r)  a}} {\Phi(q+r)-\theta}  \Big( \frac 1 r - \int_0^{-a} e^{-\Phi(q+r) z}  W^{(q)}(z)   \diff z \Big).
\end{align*}

Hence putting the pieces together we obtain that for $\theta <\Phi(q)$
\begin{multline*}
\lim_{b\uparrow\infty}\frac{Z^{(q,r)}_a(b,\theta )}{W^{(q+r)}(b)}
=\frac{r}{\Phi(q+r)-\theta }Z^{(q)}(-a,\theta )+\frac{q-\psi(\theta )}{\Phi(q+r)-\theta }Z^{(q)}(-a,\Phi(q+r)) =\tilde{Z}^{(q,r)}(-a,\theta ).
\end{multline*}

(iv) Because we can write $\overline{W}_a^{(q,r)}(b)=[Z^{(q,r)}_a(b)-1 -r\overline{W}^{(q+r)}(b)]/q$,
the result holds by (iii) and \eqref{conv_ratio_scale_functions}.

(v) By \eqref{conv_ratio_scale_functions} and \eqref{M_expression_alternative},
\begin{align} \label{ratio_Z_bar_W_limit}
\begin{split}
\lim_{b\uparrow\infty}\frac{\overline{Z}_a^{(q,r)}(b)}{W^{(q+r)}(b)}&=\lim_{b\uparrow\infty} \frac{\overline{Z}^{(q+r)}(b-a)-r\int_0^{-a}W^{(q+r)}(b-u-a)\overline{Z}^{(q)}(u) \diff u}{W^{(q+r)}(b)}\\&=e^{-\Phi(q+r)a}\left(\frac{q+r}{\Phi^2(q+r)}-r\int_0^{-a}e^{-\Phi(q+r)u}\overline{Z}^{(q)}(u) \diff u\right).
\end{split}
\end{align}
Here, applying integration by parts twice,
\begin{multline*}
	\int_0^{-a}e^{-\Phi(q+r)u}\overline{Z}^{(q)}(u)\diff u =-e^{\Phi(q+r)a}\frac{\overline{Z}^{(q)}(-a)}{\Phi(q+r)}+\frac{1}{\Phi(q+r)}\int_0^{-a}e^{-\Phi(q+r)u}Z^{(q)}(u) \diff u, \\
	=-e^{\Phi(q+r)a}\frac{\overline{Z}^{(q)}(-a)}{\Phi(q+r)}+\frac{1}{\Phi^2(q+r)}\Big[1 -e^{\Phi(q+r)a} Z^{(q)}(-a) + q \int_0^{-a}e^{-\Phi(q+r)u}W^{(q)}(u) \diff u \Big]. \end{multline*}
Hence, the right hand side of \eqref{ratio_Z_bar_W_limit} equals
\begin{align*}
& \frac{r \overline{Z}^{(q)}(-a)}{\Phi(q+r)} + \frac{1}{\Phi^2(q+r)}\Big[  q e^{-\Phi(q+r)a}  +  r Z^{(q)}(-a) - q r e^{-\Phi(q+r)a}\int_0^{-a}e^{-\Phi(q+r)u}W^{(q)}(u) \diff u \Big] \\
&=   \frac{r \overline{Z}^{(q)}(-a)}{\Phi(q+r)} + \frac{\tilde{Z}^{(q,r)}(-a)}{\Phi(q+r)}.\end{align*}
\end{proof}

\begin{lemma} \label{lemma_convergence_I_wrt_a} Fix $q \geq 0$ and $x \in \R$.

(i) We have $\lim_{a \downarrow -\infty}
I_a^{(q,r)}(x)
 = I_{-\infty}^{(q,r)}(x) $.

(ii) We have $\lim_{a\downarrow-\infty} (I_a^{(q,r)})'(x)
=(I_{-\infty}^{(q,r)})'(x)$, where it is understood for the case $q = 0$ that it goes to infinity.
\end{lemma}
\begin{proof}
(i)  It is immediate by Lemma \ref{lemma_convergence_W_a} (i).
(ii) The proof follows because, by \eqref{W^{(q)}_limit},
\begin{align*}
\frac {(W_a^{(q,r)})'(x +)} {W^{(q)}(-a)} &= \frac{W^{(q)\prime}((x-a)+)}{W^{(q)}(-a)}+rW^{(q+r)}(0)\frac{W^{(q)}(x-a)}{W^{(q)}(-a)}+r\int_0^xW^{(q+r) \prime}(x-y)\frac{W^{(q)}(y-a)}{W^{(q)}(-a)}\diff y \\
&\xrightarrow{a \downarrow -\infty}\Phi(q)e^{\Phi(q)x}+rW^{(q+r)}(0)e^{\Phi(q)x}+r\int_0^xe^{\Phi(q)y}W^{(q+r) \prime}(x-y) \diff y,
\end{align*}
which equals $Z^{(q+r)\prime}(x,\Phi(q)) = \Phi(q) Z^{(q+r)}(x, \Phi(q)) + r W^{(q+r)}(x)$ by integration by parts. 
\end{proof}

\begin{lemma} \label{lemma_convergence_I_wrt_b} Fix $q \geq 0$ and $a < 0$.
(i) We have
\begin{align*} 
\lim_{b \uparrow \infty }\frac {I_a^{(q,r)} (b)} {W^{(q+r)}(b)} &=\frac {Z^{(q)}  (-a, \Phi(q+r))} {W^{(q)}(-a)}  - \frac r {\Phi(q+r)} = \frac {
Z^{(q)\prime}  (-a, \Phi(q+r))
} {W^{(q)}(-a) \Phi(q+r)}.
\end{align*}
(ii) For $0 \leq \theta  < \Phi(q)$,
\begin{align*}
\lim_{b \uparrow \infty}	\frac {J_a^{(q,r)} (b,\theta )} {W^{(q+r)}(b)} = \tilde{Z}^{(q,r)}(-a,\theta ) - \frac {r Z^{(q)}(-a, \theta)} {\Phi(q+r)}, \end{align*}
where in particular
 \begin{align*}
 	\lim_{b\uparrow\infty}\frac{J_a^{(q,r)} (b)}{W^{(q+r)}(b)}=\frac{q Z^{(q)}(-a,\Phi(q+r)) }{\Phi(q+r)}. 
 \end{align*}
(iii) We have
\begin{align*}
\lim_{b \uparrow \infty} \frac {K_a^{(q,r)}(b)}   {W^{(q+r)}(b)}
&=  \frac 1 {\Phi(q+r)} \Big(  \tilde{Z}^{(q,r)}(-a )- \psi'(0+) {Z^{(q)}(-a, \Phi(q+r))} \Big).
\end{align*}
 
\end{lemma}
\begin{proof} (i) By Lemma \ref{lemma_convergence_W_a} (ii) and \eqref{conv_ratio_scale_functions},
\begin{align*}
\frac {I_a^{(q,r)} (b)} {W^{(q+r)}(b)} = \frac {W^{(q,r)}_a(b)} {W^{(q+r)}(b)W^{(q)}(-a)}  - r  \frac {\overline{W}^{(q+r)} (b)} {W^{(q+r)}(b)} \xrightarrow{b \uparrow \infty}\frac {Z^{(q)}  (-a, \Phi(q+r))} {W^{(q)}(-a)}  - \frac r {\Phi(q+r)}. 
\end{align*}

(ii) By Lemma \ref{lemma_convergence_W_a} (iii) and \eqref{conv_ratio_scale_functions}, we have
\begin{align*}
\frac {J_a^{(q,r)} (b,\theta )} {W^{(q+r)}(b)} = \frac {Z^{(q,r)}_a(b,\theta )} {W^{(q+r)}(b)} - \frac {r Z^{(q)} (-a, \theta) \overline{W}^{(q+r)}(b)} {W^{(q+r)}(b)} \xrightarrow{b \uparrow \infty} \tilde{Z}^{(q,r)}(-a,\theta ) - \frac {r Z^{(q)}(-a, \theta)} {\Phi(q+r)}. \end{align*}
The case $\theta = 0$ holds by \eqref{Z_q_r}.

(iii) 
By Lemma \ref{lemma_convergence_W_a} (ii) and (v) and \eqref{conv_ratio_scale_functions}, 
%
\begin{align*}
&\lim_{b \uparrow \infty} \frac {K_a^{(q,r)}(b)}   {W^{(q+r)}(b)} =  - r \frac { l^{(q)}(-a)  }  {\Phi(q+r)}   +  \lim_{b \uparrow \infty} \frac {\overline{Z}_a^{(q,r)}(b) }  {W^{(q+r)}(b)}- \psi'(0+) \lim_{b \uparrow \infty} \frac {\overline{W}_a^{(q,r)}(b)}  {W^{(q+r)}(b)} \\
&=  \frac 1 {\Phi(q+r)} \Big( r \psi'(0+) \overline{W}^{(q)}(-a)  +  \tilde{Z}^{(q,r)}(-a )- \psi'(0+) \frac{\Phi(q+r)\tilde{Z}^{(q,r)}(-a )-r}{q} \Big) \\
&=  \frac 1 {\Phi(q+r)} \Big( r \psi'(0+) \frac {Z^{(q)}(-a)} q +  \tilde{Z}^{(q,r)}(-a )- \psi'(0+) \frac{\Phi(q+r)\tilde{Z}^{(q,r)}(-a )}{q} \Big) \\
&=  \frac 1 {\Phi(q+r)} \Big(  \tilde{Z}^{(q,r)}(-a)- \psi'(0+) {Z^{(q)}(-a, \Phi(q+r))} \Big).
\end{align*}

%
\end{proof}

\subsection{Proof of Corollary \ref{cor_div}}
(i) In view of Theorem \ref{prop_dividends}, it is immediate upon taking $a \downarrow -\infty$ by monotone convergence and Lemma \ref{lemma_convergence_I_wrt_a} (i). The convergence \eqref{I_infinity_conv} is confirmed in Lemma \ref{lemma_convergence_I_wrt_a} (i).

(ii) 
Similarly, it suffices to take $b \uparrow \infty$.  In addition, by Lemma \ref{lemma_convergence_I_wrt_b} (i) and \eqref{conv_ratio_scale_functions},
	 	 \begin{align*}
 \lim_{b \uparrow \infty}  \frac {\overline{\overline{W}}^{(q+r)}(b) } {I_a^{(q,r)}(b)} 
 =  \lim_{b \uparrow \infty}  \frac  {\overline{\overline{W}}^{(q+r)}(b) }  {W^{(q+r)}(b)}  \lim_{b \uparrow \infty}  \frac {W^{(q+r)}(b)} {I_a^{(q,r)}(b)}
= \frac {1} {\Phi(q+r)} \frac {W^{(q)}(-a)}  { Z^{(q)\prime}(-a, \Phi(q+r))}. 
	 \end{align*}

(iii) We shall show for the case $q > 0$.  The case $q = 0$ holds by monotone convergence. 
By monotone convergence, it suffices to take $b \uparrow \infty$ in (i).  By \eqref{conv_ratio_scale_functions}, this boils down to computing
	 \begin{align*}
\lim_{b \uparrow \infty}  \frac {\overline{\overline{W}}^{(q+r)}(b)} {I_{-\infty}^{(q,r)}(b)}  =  \frac {1} {\Phi^2(q+r)}  \lim_{b \uparrow \infty}  \frac {W^{(q+r)}(b)} {I_{-\infty}^{(q,r)}(b)}. 
	 \end{align*}
%
In addition, by \eqref{W^{(q)}_limit},
\begin{multline*}
\frac {I_{-\infty}^{(q,r)}(b)}  {W^{(q+r)}(b)}
= \frac {e^{\Phi(q) b}+ r \int_0^b e^{\Phi(q) z } W^{(q+r)} (b-z)  \diff z - r \overline{W}^{(q+r)}(b)} {W^{(q+r)}(b)}\\
\xrightarrow{b \uparrow \infty} r \Big( \int_0^\infty e^{\Phi(q) z } e^{- \Phi(q+r) z}  \diff z - \frac 1 {\Phi(q+r)} \Big) 
= \frac {r \Phi(q)} {(\Phi(q+r) - \Phi(q)) \Phi(q+r)}.
\end{multline*}

\subsection{Proof of Corollary \ref{corollary_g}} 
(i) In view of Theorem \ref{proposition_upcrossing_time}, it is immediate upon taking $a \downarrow -\infty$ by monotone convergence and Lemma \ref{lemma_convergence_I_wrt_a} (i).  
(ii) It is immediate by setting $q = 0$ and $\Phi(q) = 0$ in (i) and noticing that in this case  $I_{-\infty}^{(0,r)}(x) = 1$ uniformly in $x$.

\subsection{Proof of Corollary \ref{corollary_h_a}}

(i) We shall show for the case $0 \leq \theta  < \Phi(q)$; the cases $\theta  \geq \Phi(q)$ holds by analytic continuation.  In view of Theorem \ref{proposition_laplace}, by monotone convergence, it suffices to take $b \uparrow\infty$.
By Lemma \ref{lemma_convergence_I_wrt_b} (i) and (ii), we have the claim.

(ii) By taking $\theta = 0$ and  $q=0$ in (i) we obtain the claim in view of \eqref{J_tilde_q_zero}.


\subsection{Proof of Corollary \ref{corollary_creeping}}
(i) By \eqref{laplace_in_terms_of_z} and \eqref{creeping_identity}, and monotone convergence,
\begin{align} \label{creeping_formula}
\begin{split}
\lim_{\theta \uparrow\infty}&\left(Z^{(q)}(x-a,\theta )-\frac{Z^{(q)}(-a,\theta )}{W^{(q)}(-a)}W^{(q)}(x-a)\right)\\
&=\lim_{\theta \uparrow\infty} \E_{x-a}\left(e^{-q\tau_0^-+\theta X(\tau_0^-)};\tau_0^-<\infty\right)-\frac{W^{(q)}(x-a)}{W^{(q)}(-a)} \lim_{\theta \uparrow\infty} \E_{-a}\left(e^{-q\tau_0^-+\theta X(\tau_0^-)};\tau_0^-<\infty\right)\\
&=\E_{x-a}\left(e^{-q\tau_0^-}; X(\tau_0^-)=0, \tau_0^-<\infty\right)-\frac{W^{(q)}(x-a)}{W^{(q)}(-a)}\E_{-a}\left(e^{-q\tau_0^-}; X(\tau_0^-)=0, \tau_0^-<\infty\right)\\
&=\frac{\sigma^2}{2}\left[\left(W^{(q)\prime}(x-a)-\Phi(q)W^{(q)}(x-a)\right)-\frac{W^{(q)}(x-a)}{W^{(q)}(-a)}\left(W^{(q)\prime}(-a)-\Phi(q)W^{(q)}(-a)\right)\right]\\
&=C_{-a}^{(q)}(x-a).
\end{split}
\end{align}
This implies that
\begin{multline*}
\lim_{\theta \uparrow\infty}\hat{J}_a^{(q,r)}(x,\theta ) = \lim_{\theta \uparrow\infty} \mathcal{M}_a^{(q,r)}\left(Z^{(q)}(x,\theta )-\frac{Z^{(q)}(-a,\theta )}{W^{(q)}(-a)}W^{(q)}(x)\right)\\
=C_{-a}^{(q)} (x-a)+r\int_0^xW^{(q+r)}(x-y)C_{-a}^{(q)} (y-a) \diff y=
\mathcal{M}_a^{(q,r)} C_{-a}^{(q)} (x).
\end{multline*}
Here, the limit can go into the integral because, by \eqref{creeping_formula}, $\sup_{0 \leq y \leq x}| Z^{(q)}(y-a,\theta )- Z^{(q)}(-a,\theta )W^{(q)}(y-a) / {W^{(q)}(-a)}| \leq 1 +  {W^{(q)}}(x-a) / W^{(q)}(-a)$ uniformly in $\theta \geq 0$.

Hence taking $\theta  \uparrow \infty$ in Theorem  \ref{proposition_laplace}, we have
\begin{align*}
w(x,a,b ) = \mathcal{M}_a^{(q,r)} C_{-a}^{(q)} (x)  - \frac {I_a^{(q,r)}(x)}   {I_a^{(q,r)}(b)} \mathcal{M}_a^{(q,r)} C_{-a}^{(q)}(b). 
\end{align*}
  Because
\begin{align*} 
\mathcal{M}_a^{(q,r)} C_{-a}^{(q)} (x)  =  C_a^{(q,r)}(x)  - \frac{\sigma^2}{2}I_a^{(q,r)} (x) W^{(q)\prime} (-a),
\end{align*}
we have the claim. 

(ii) By \eqref{conv_ratio_scale_functions} and \eqref{W^{(q)}_limit}, we have 
\begin{align*}
\lim_{b\uparrow \infty}\frac{\mathcal{M}_a^{(q,r)} W^{(q)\prime}(b)}{W^{(q+r)}(b)}&=\lim_{b\uparrow \infty}\frac{1}{W^{(q+r)}(b)}\left(W^{(q)\prime}(b-a)+r\int_0^bW^{(q+r)}(b-y)W^{(q)\prime}(y-a)\diff y\right)\\&=r\int_0^{\infty}e^{-\Phi(q+r)y}W^{(q)\prime}(y-a) \diff y,
\end{align*}
where integration by parts gives
\begin{align*}
\int_0^{\infty}e^{-\Phi(q+r)y}  W^{(q)\prime}(y-a)  \diff y &= e^{-\Phi(q+r)a}  \int_{-a}^{\infty}e^{-\Phi(q+r)z}  W^{(q)\prime}(z)  \diff z \\
&= - W^{(q)} (-a) + \frac 1 r \Phi(q+r) Z^{(q)} (-a, \Phi(q+r)) = \frac 1 r  Z^{(q) \prime} (-a, \Phi(q+r)).
\end{align*}
This together with \eqref{conv_ratio_scale_functions} shows
\begin{align*}
\lim_{b \uparrow \infty}\frac {C_a^{(q,r)}(b)} {W^{(q+r)}(b)} 
= \frac{\sigma^2}{2}\left(Z^{(q) \prime} (-a, \Phi(q+r)) - r \frac{W^{(q)\prime}(-a)} {\Phi(q+r) }\right).
\end{align*}
Now the proof is complete because, by Lemma \ref{lemma_convergence_I_wrt_b} (i), 
\begin{align*}
\frac {C_a^{(q,r)}(b)}   {I_a^{(q,r)}(b)}   
\xrightarrow{b \uparrow \infty}&     \frac{\sigma^2}{2} \Big[ Z^{(q) \prime} (-a, \Phi(q+r)) - r \frac {W^{(q)\prime}(-a)} {\Phi(q+r)}  \Big] \frac {W^{(q)}(-a) \Phi(q+r)}  {
Z^{(q)\prime}  (-a, \Phi(q+r))
}  \\
&=    W^{(q)}(-a) \frac{\sigma^2}{2} \Big[  \Phi(q+r)  - r \frac {W^{(q)\prime}(-a)  }  {
Z^{(q)\prime}  (-a, \Phi(q+r))
}  \Big] .
\end{align*}

\subsection{Proof of Corollary \ref{corollary_overshoot}}

For $\theta > 0$ and $x \in \R$,
\begin{align} \label{Z_v_derivative}
\frac {\partial Z^{(q)}(x, \theta )} {\partial \theta } = x Z^{(q)} (x,\theta ) - e^{\theta x} \Big[\psi'(\theta ) \int_0^x e^{-\theta  z} W^{(q)}(z) \diff z + (q-\psi(\theta )) \int_0^x e^{-\theta z} z W^{(q)}(z) \diff z \Big].
\end{align}
Because  integration by parts gives $\int_0^{x}  y W^{(q)} (y) \diff y 
= x \overline{W}^{(q)} (x) - \overline{\overline{W}}^{(q)} (x) $,
\begin{align*}
\lim_{\theta  \downarrow 0}\frac {\partial Z^{(q)}(x, \theta )} {\partial \theta }  &=x  \left( 1 + q \overline{W}^{(q)}(x)	\right)  - \psi'(0+) \overline{W}^{(q)}(x) - q \int_0^{x} z W^{(q)}(z) \diff z = l^{(q)}(x). 
\end{align*}
Hence,
\begin{align*}
&\lim_{\theta  \downarrow 0}\frac {\partial Z^{(q,r)}_a(x,\theta )} {\partial \theta }   =   \lim_{\theta  \downarrow 0}\frac {\partial Z^{(q)} (x-a, \theta ) } {\partial \theta }  +r \int_0^x W^{(q+r)} (x-y) \lim_{\theta  \downarrow 0}\frac {\partial Z^{(q)} (y-a, \theta )} {\partial \theta }   \diff y 
=l_a^{(q,r)}(x).
\end{align*}
Hence, $K_a^{(q,r)}(x) = \lim_{\theta \downarrow 0}  (\partial J_a^{(q,r)}(x, \theta) / {\partial \theta}) $ and the result holds by Theorem \ref{proposition_laplace}.

\subsection{Proof of Corollary \ref{corollary_overshoot_limit}}  In view of Corollary \ref{corollary_overshoot}, by monotone convergence, it suffices to take $b \uparrow\infty$. 
Now the result holds by Lemma \ref{lemma_convergence_I_wrt_b} (i) and (iii).

\subsection{Proof of Corollary \ref{corollary_L_tilde_r_P}} 

For the case $q > 0$, in view of Proposition \ref{prop_f_tilde_p}, it is immediate upon taking $a \downarrow -\infty$ by monotone convergence and Lemma \ref{lemma_convergence_I_wrt_a} (i) and (ii).  The case $q = 0$ holds by monotone convergence upon taking $q \downarrow 0$. 
 
 \subsection{Proof of Corollary \ref{corollary_f_tilde_S}} 

For the case $q > 0$, in  view of Proposition \ref{prop_f_tilde_S}, it is immediate upon taking $a \downarrow -\infty$ by monotone convergence and Lemma \ref{lemma_convergence_I_wrt_a} (i) and (ii). The case $q = 0$ holds by monotone convergence upon taking $q \downarrow 0$.

\subsection{Proof of Corollary \ref{cor_overshoot_reflected}} We shall take $\lim_{\theta  \downarrow 0} \partial \tilde{h} (x,a,b, \theta ) / {\partial \theta }$ in Proposition \ref{prop_g_reflection_above}. 
By \eqref{Z_v_derivative}, it can be confirmed that
\begin{align*}
\lim_{\theta  \downarrow 0}\frac \partial {\partial \theta } Z^{(q)\prime}(x, \theta ) = \lim_{\theta  \downarrow 0}\frac \partial {\partial x} \frac \partial {\partial \theta }  Z^{(q)}(x, \theta )
= Z^{(q)}(x)  - \psi'(0+) W^{(q)}(x) = \frac \partial {\partial x}\Big(\lim_{\theta  \downarrow 0}\frac \partial {\partial \theta } Z^{(q)}(x, \theta ) \Big).
\end{align*}
Hence, $(K_a^{(q,r)})'(x) = \lim_{\theta \downarrow 0}  (\partial (J_a^{(q,r)})'(x, \theta) / {\partial \theta}) $ and by modifying the proof of Corollary \ref{corollary_overshoot}, we have the result.


 \subsection{Proof of Corollary \ref{corollary_dividends_limit}} For the case $q > 0$, in  view of Proposition 
 \ref{prop_f_hat}, by monotone convergence, it is immediate by Lemma \ref{lemma_convergence_I_wrt_b} (ii) and \eqref{conv_ratio_scale_functions}. The case $q = 0$ holds by monotone convergence upon taking $q \downarrow 0$.
 

\subsection{Proof of Corollary \ref{corollary_capital_injection_limit}}

In view of Proposition 
 \ref{prop_capital_injection}, by monotone convergence, it suffices to take $b \uparrow \infty$. 
 Using Lemma \ref{lemma_convergence_I_wrt_b} (ii) and (iii), we have that
\begin{align*}
\lim_{b\uparrow\infty}\frac{K_a^{(q,r)}(b)}{J_a^{(q,r)}(b)}
&=\frac{\tilde{Z}^{(q,r)}(-a)- \psi'(0+) {Z^{(q)}(-a, \Phi(q+r))}}{qZ^{(q)}(-a,\Phi(q+r))}.
\end{align*}
Hence,
 \begin{align*}
 \lim_{b\uparrow\infty}\frac{H_a^{(q,r)}(b)}{J_a^{(q,r)}(b)}&=\lim_{b\uparrow\infty}\frac{K_a^{(q,r)}(b)}{J_a^{(q,r)}(b)}-\frac{l^{(q)}(-a)}{Z^{(q)}(-a)} =\frac{\tilde{Z}^{(q,r)}(-a)- \psi'(0+) {Z^{(q)}(-a, \Phi(q+r))}}{qZ^{(q)}(-a,\Phi(q+r))}-\frac{l^{(q)}(-a)}{Z^{(q)}(-a)}\\
 	&=\frac{1}{q}\left(\frac{\tilde{Z}^{(q,r)}(-a)}{Z^{(q)}(-a,\Phi(q+r))}-\frac{q\overline{Z}^{(q)}(-a)+\psi'(0)}{Z^{(q)}(-a)} \right).
 \end{align*}
Hence putting the pieces together we have
	\begin{align*}
	\E_x&\left(\int_{[0, \infty)}e^{-qt} \diff R_r^a(t)\right)=\frac{1}{q}\left(\frac{\tilde{Z}^{(q,r)}(-a )}{Z^{(q)}(-a,\Phi(q+r))}-\frac{q\overline{Z}^{(q)}(-a)+\psi'(0)}{Z^{(q)}(-a)}\right)J_a^{(q,r)}(x)-H_a^{(q,r)}(x),
	\end{align*}
	which equals
	\begin{align*}
	&
	\frac{1}{q} \left(\frac{rZ^{(q)} (-a) + qZ^{(q)} (-a , \Phi (q+r))}{\Phi (q + r)Z^{(q)}(-a , \Phi (q + r) )}- 
	\frac{q\overline{Z}^{(q)} (-a) + \psi^{\prime} (0+) }{Z^{(q)} (-a)}\right) \Big( Z_{a}^{(q, r)} (x) - r Z^{(q)}(-a) \overline{W}^{(q + r)} (x) \Big) \\
	&- \Big( l^{(q,r)}_a(x)  - \frac{l^{(q)}(-a)}{Z^{(q)}(-a)}Z_{a}^{(q, r)} (x) \Big) \\
	&=\left(\frac{rZ^{(q)} (-a) }{q \Phi (q + r)Z^{(q)}(-a , \Phi (q + r) )} + \frac 1 {\Phi(q+r)}\right)\Big( Z_{a}^{(q, r)} (x) - r Z^{(q)}(-a) \overline{W}^{(q + r)} (x) \Big) \\
	&+ \Big(\overline{Z}^{(q)} (-a) + \frac {\psi^{\prime} (0+)} q \Big) r \overline{W}^{(q + r)} (x)  -  l^{(q,r)}_a(x)  + \frac{l^{(q)}(-a) - \overline{Z}^{(q)} (-a) - \psi^{\prime} (0+)/q }{Z^{(q)}(-a)}Z_{a}^{(q, r)} (x) 
	\end{align*}
	Here, we have
	\begin{align*}
	\frac{l^{(q)}(-a) - \bar{Z}^{(q)} (-a) - \psi^{\prime} (0+)/q }{Z^{(q)}(-a)} = - \frac {\psi'(0+)} q
	\end{align*}
	and
	\begin{align*}
	 l^{(q,r)}_a(x) 
	&=   \overline{Z}^{(q,r)}_a (x)-\frac{\psi^\prime (0+)}{q}Z^{(q,r)}_a(x) +\frac{\psi^\prime(0+)}{ q}  + r \frac{\psi^\prime(0+)}{ q}  \overline{W}^{(q + r)}(x).
	\end{align*}
	Substituting these, we have
	\begin{align*}
		&\E_x \left(\int_{[0, \infty)}e^{-qt} \diff R_r^a(t)\right) \\&= \left(\frac{rZ^{(q)} (-a) }{q \Phi (q + r)Z^{(q)}(-a , \Phi (q + r) )} + \frac 1 {\Phi(q+r)}\right) \Big( Z_{a}^{(q, r)} (x) - r Z^{(q)}(-a) \overline{W}^{(q + r)} (x) \Big) \\
	&+ \Big(\overline{Z}^{(q)} (-a) + \frac {\psi^{\prime} (0+)} q \Big) r \overline{W}^{(q + r)} (x)  \\
	&- \Big( \overline{Z}^{(q,r)}_a (x)-\frac{\psi^\prime (0+)}{q}Z^{(q,r)}_a(x) +\frac{\psi^\prime(0+)}{ q}  + r \frac{\psi^\prime(0+)}{ q}  \overline{W}^{(q + r)}(x) \Big)  - \frac {\psi'(0+)} q Z^{(q,r)}_a(x) \\
	 		&= \left(\frac{rZ^{(q)} (-a) }{q \Phi (q + r)Z^{(q)}(-a , \Phi (q + r) )} + \frac 1 {\Phi(q+r)}\right) \Big( Z_{a}^{(q, r)} (x) - r Z^{(q)}(-a) \overline{W}^{(q + r)} (x) \Big) \\
	&+ r \overline{Z}^{(q)} (-a)  \overline{W}^{(q + r)} (x)  - \Big( \overline{Z}^{(q,r)}_a (x) +\frac{\psi^\prime(0+)}{ q}   \Big).
	\end{align*}

\appendix

	 \end{document}